\documentclass{iopart}
\usepackage[utf8]{inputenc} 
\usepackage[colorlinks=true, citecolor=blue, filecolor=black,linkcolor=blue, urlcolor=blue, hypertexnames=false]{hyperref}
\usepackage{color}

\usepackage{array,multirow}
\usepackage{enumitem}

\usepackage{rotating}
\usepackage{multirow}
\usepackage{titlesec}

\usepackage{afterpage}
\usepackage{pdflscape}

\usepackage{tikz}
\usetikzlibrary{positioning,shapes,calc,3d}
\usepackage{pgfplots}

\expandafter\let\csname equation*\endcsname\relax
\expandafter\let\csname endequation*\endcsname\relax
\usepackage{amsmath}
\usepackage{amssymb}
\usepackage{amsmath}
\usepackage{amsfonts}
\usepackage{amsthm}
\usepackage{iopams}
\usepackage{subdepth}

\usepackage{graphicx}
\usepackage{geometry}
\usepackage{subcaption}
\usepackage[ruled,linesnumbered]{algorithm2e}
\theoremstyle{definition}
\newtheorem{theorem}{Theorem}[section]
\newtheorem{proposition}[theorem]{Proposition}
\newtheorem{lemma}[theorem]{Lemma}

\newtheorem{corollary}[theorem]{Corollary}

\theoremstyle{definition}

\newtheorem{example}[theorem]{Example}

\theoremstyle{remark}

\newcommand{\R}{\mathbb{R}}
\newcommand{\E}{\mathbb{E}}

\DeclareMathOperator{\trace}{trace}

\DeclareMathOperator{\Dkl}{D_\mathrm{KL}}

\newcommand{\reals}{\mathbb{R}}
\newcommand{\Zhat}{\widehat{\calZ}}
\newcommand{\Hhat}{\widehat{H}}
\newcommand{\That}{\widehat{T}}

\newcommand{\calZ}{\mathcal{Z}}

\newcommand{\ghat}{\hat{g}}
\newcommand{\res}{\mathcal{R}}

\newcommand{\Hell}{\mathrm{D}_{\text{Hell}}}

\newcommand{\pihat}{\widehat{\pi}}
\newcommand{\gbar}{\bar{g}}

\newcommand{\calP}{\mathcal{P}}

\newcommand{\phihat}{\widehat{\phi}}

\newcommand{\apdf}{\psi}
\newcommand{\apdfhat}{\hat{\psi}}

\renewcommand{\d}{\mathrm{\,d}}

\usepackage[textsize=tiny,textwidth=0.9in]{todonotes}

\begin{document}

\title[Prior normalization for certified dimension reduction]{Prior normalization for certified likelihood-informed subspace detection of Bayesian inverse problems}

\author{Tiangang Cui$^{1}$, Xin Tong$^{2}$, Olivier Zahm$^{3}$}
\address{$^1$School of Mathematics, Monash University, VIC 3800, Australia}
\address{$^2$National University of Singapore, Department of Mathematic, Singapore}
\address{$^3$Univ. Grenoble Alpes, Inria, CNRS, Grenoble INP, LJK, 38000 Grenoble, France}
\ead{\mailto{tiangang.cui@monash.edu}, \mailto{mattxin@nus.edu.sg}, \mailto{olivier.zahm@inria.fr}}

\begin{abstract}
Markov Chain Monte Carlo (MCMC) methods form one of the algorithmic foundations of Bayesian inverse problems. The recent development of likelihood-informed subspace (LIS) methods offers a viable route to designing efficient MCMC methods for exploring high-dimensional posterior distributions via exploiting the intrinsic low-dimensional structure of the underlying inverse problem. However, existing LIS methods and the associated performance analysis often assume that the prior distribution is Gaussian. This assumption is limited for inverse problems aiming to promote sparsity in the parameter estimation, as heavy-tailed priors, e.g., Laplace distribution or the elastic net commonly used in Bayesian LASSO, are often needed in this case. To overcome this limitation, we consider a prior normalization technique that transforms any non-Gaussian (\emph{e.g.} heavy-tailed) priors into standard Gaussian distributions, which makes it possible to implement LIS methods to accelerate MCMC sampling via such transformations. We also rigorously investigate the integration of such transformations with several MCMC methods for high-dimensional problems. Finally, we demonstrate various aspects of our theoretical claims on two nonlinear inverse problems.
\end{abstract}

\section{Introduction}
Many mathematical modeling problems need to solve an inverse problem that aims to recover an unknown parameter $x$ from indirect and noisy data $y$ via the parameter-to-observable map
\begin{equation}
    \label{eqn:ip}
    y=G(x)+\eta,
\end{equation}
where $x\mapsto G(x)$ is a forward model and $\eta$ is the observation noise. 
In the Bayesian framework, one common way to solve the inverse problem is to draw random variables from the posterior distribution
\begin{equation}
\label{eqn:bayes}
\pi^y(x) = \frac{1}{\calZ} \, f(x;y)\pi^0(x) , \quad \calZ = \int_{\reals^d}f(x;y) \pi^0(x)\d x ,
\end{equation}
where $x\mapsto f(x;y)$ is the \emph{likelihood} function of obtaining the data $y$ for a given parameter $x$ using the parameter-to-observable map \eqref{eqn:ip}, $x\mapsto \pi^0(x)$ is the \emph{prior} density that encodes prior knowledge of the parameter, and $\calZ$ is the \emph{normalizing constant} that is often unknown.

In this work, we consider inverse problems with high-dimensional parameter $x\in\R^d$ where $d\gg1$. Such problems are typically encountered when $x$ arises from the discretization of spatially or temporally heterogeneous objects, for instance, tissue properties in medical imaging. The likelihood function $x\mapsto f(x;y)$ represents the measurement process of observing the data $y$, which involves a forward model $x\mapsto G(x)$ typically derived from differential equations or integral equations. 
There are many ways to set up the prior $\pi^0$. One classical choice is choosing $\pi^0$ as a Gaussian process, which is closely related to the Tikhonov regularization \cite{Stuart10}. There are some recent works trying to use heavy-tailed priors, e.g. Laplace distribution, TV-Gaussian, Student's $t$-distribution and Cauchy distribution \cite{chada2021cauchy,hosseini2017well,markkanen2019cauchy,sullivan2017well,yao2016tv}, to promote sparsity in the solution to the inverse problems.
These heavy-tailed distributions are also natural choices for the modeling of rare events \cite{majda2015intermittency,majda2019simple}. Moreover, some of these distributions, e.g. Student's $t$ and Cauchy, are infinitely-devisable \cite{kallenberg1997foundations}. This property makes them good choices for temporal and spatial models \cite{sullivan2017well,unser2014introduction} and is used in imaging applications for their edge-preserving properties \cite{hosseini2017well}.

Markov Chain Monte Carlo (MCMC) algorithms are popular workhorses to sample from the posterior distribution. However, the performance of MCMC algorithms can be constrained by several features of  large-scale inverse problems. Firstly, many MCMC algorithms can have degenerating efficiency with increasing dimensionality of the parameter $x$. See \cite{roberts2001optimal} and references therein for detailed discussions. Secondly, the landscape of the posterior $\pi^y$, which can be largely controlled by the landscape of the prior $\pi^0$, also affects the convergence of MCMC algorithms. In particular, when the invariant distribution is strongly log-concave outside of a compact set or has sub-Gaussian tails, MCMC algorithms can have rapid convergence to the invariant distribution \cite{chen2018robust,ma2019sampling}. However, such assumptions are invalid for problems with heavy-tailed priors. Together with  high-dimensionality, it can be particularly challenging to design efficient MCMC algorithms for solving inverse problems equipped with heavy-tailed priors. 

In this work, we present a combined treatment to address these challenges via the likelihood-informed dimension reduction \cite{cui2014likelihood,cui2021unified,zahm2018certified}. The fundamental idea is to identify the most likelihood-informed directions in the high-dimensional parameter space, along which the prior and the posterior differ the most. The resulting likelihood-informed subspace (LIS) provides effective reduced-dimensional approximations to the original high-dimensional posterior \cite{EtAl1Cui6,cui2014likelihood}. Furthermore, for problems equipped with Gaussian priors, the parameter space decomposition offered by such approximations also opens the door to accelerating standard MCMC algorithms targeting the original posterior.
Assuming the prior to be Gaussian or log-concave distributions, rigorous error bounds are derived in \cite{cui2021unified,zahm2018certified} to certify that the dimension reduction leads to accurate approximations. Such assumptions, however, cannot be satisfied for heavy-tailed priors. In this paper, we aim to derive a similar certified dimension reduction technique for inverse problems equipped with heavy-tailed priors.

The path we take is a prior normalization technique that transforms a challenging prior distribution into a Gaussian distribution. In particular, we consider a bijective map $T:z\mapsto T(z),$ such that the random vector $X=T(Z)$ follows the prior density $\pi^0$ whenever $Z$ is a standard Gaussian vector $Z$ with density $\phi^0(z) \propto \exp(-\tfrac12 \|z\|^2)$. In other words,  the prior density is the pushforward image of a standard Gaussian under the map $T$, i.e., $\pi^0=T_\sharp\phi^0$. Because $T$ is a bijection, we can pullback the posterior density to obtain a density in the \emph{reference coordinate} $z$ equipped with a standard Gaussian prior. The pullback posterior density, $\phi^y= T^\sharp\pi^y$, can be expressed as 
\[
\phi^y(z)= \frac{1}{\calZ} \, g(z;y) \phi^0(z),
\qquad g(z;y)=f(T(z);y) ,
\]
where $\calZ$ is the same normalizing density as in \eqref{eqn:bayes}.
Naturally, we can implement MCMC algorithms with $\phi^y$ as the invariant distribution to obtain posterior samples in the reference coordinate, and then transform them via $T$ to obtain posterior samples in the original coordinate. 

Since the prior in the reference coordinate $\phi^0(z)$ is Gaussian, the above prior normalization may make the tails of the transformed posterior $\phi^y(z)$ easier to explore with standard MCMC methods. For instance, if the original likelihood $x\mapsto f(x;y)$ is constant outside a compact set, then the transformed posterior $\phi^y$ is strongly-log-concave outside a compact domain, and so it meets the fast MCMC convergence criteria described by \cite{ma2019sampling}.

The prior normalization has been considered in various areas of statistics, see \cite{lemaire2013structural,nataf1962determination,lebrun2009innovating}. For Bayesian inverse problems, although linear transformation is widely used as a preconditioner to accelerate inference algorithms, there are only a few existing works that exploit prior normalization to accelerate MCMC for problems with heavy-tailed priors. In particular,  \cite{fleischer2007transformations} discusses how to use transformation to sample a one-dimensional distribution which is  multimodal, and \cite{chen2018robust,wang2017bayesian} discuss how to implement the preconditioned Crank-Nicholson (pCN) algorithm and the random-then-optimize algorithm with prior normalization. To the best of our knowledge, there is a gap between these algorithmic developments and the rigorous analysis of the impact of prior normalization on high-dimensional MCMC algorithms. In addition, there is also a lack of understanding on how to mitigate errors caused by using an approximate transformation $\That$ in the situation where $T$ is not accessible nor tractable to compute.

This article contributes to the above issues by systematically investigating the usage of prior normalization for high-dimensional MCMC algorithms.
First in Section \ref{sec:Tconstruct}, we show how to construct the transformation maps $T(z)$ and we analyze its asymptotic behaviour when $z\rightarrow\infty$ for various heavy-tails prior.
Then in Section \ref{sec:MCMC}, we show how to identify the intrinsic low-dimensional structure of the high-dimensional  transformed posterior $\phi^y(z)$.
This naturally leads to the design of scalable sampling methods that allocate computation resources to the most effective dimensions, while maintaining the original posterior as the invariant distribution. 
Section \ref{sec:theory} provides rigorous bounds for the approximation errors and computational inefficiency caused by dimensional reduction and by the usage of approximate transformation. Section \ref{sec:numerics} provides two numerical examples demonstrating the efficiency of the proposed methods.

\section{Prior normalization}
\label{sec:Tconstruct}

\subsection{Product form priors}
\label{sec:diagonalprior}
In this work, we consider prior densities that can be effectively expressed in a product-form of
$$\pi^0(x)=\prod_{i=1}^d \pi^0_i(x_i),$$
where $\pi^0_i$ denotes the $i$-th marginal density. 
Such product-form priors naturally appear when a random fields $f$ is defined via an expansion $f=\sum_{i\geq1} X_i f_i$ on a deterministic function basis $(f_1,f_2,\hdots)$ and where $X_i\sim \pi^0_i$ are independent random variables, see for instance the Besov random fields \cite{dashti2012besov,saksman2009discretization}.
Denoting the cumulative distribution function (CDF) of $\pi^0_i$ by $\calP_i^0(x)=\int_{-\infty}^x \pi_i^0(t)\d t$ and the CDF of the univariate standard Gaussian density $\phi^0(z_i)$ by $\Phi^0(z)=\int_{-\infty}^z \phi^0(t)\d t$, the diagonal transformation 
\begin{equation}
    \label{eqn:prod}
    T(z) = \begin{pmatrix} T_1(z_1) \\\vdots \\ T_d(z_d) \end{pmatrix}
 ,\quad T_i(z_i) = (\calP_{i}^0)^{-1}\circ \Phi^0(z_i) ,
\end{equation} 
pushes forward the reference standard Gaussian density $\phi^0$ to the prior $\pi^0$.
Here, $(\calP_{i}^0)^{-1}$ is the inverse of $\calP_i^0$. 
It is well known that $T_i$ corresponds to the \emph{optimal transport} from the reference density $\phi^0$ to the prior density $\pi_i^0$, see \cite[Theorem 2.18]{villani2009optimal}. Assuming that $\pi^0(x)>0$ for all $x\in\R^d$, the map $T$ is everywhere differentiable and
$$
T_i'(z_i) = \frac{\phi^0(z_i)}{\pi_i^0(T_i(z_i))} > 0.
$$
For many classical distributions, e.g. Laplace and Student's $t$, the CDF $\calP_i^0$ is known analytically so that the transformation $T$ and its inverse can be evaluated up to machine precision.
Otherwise, a numerical approximation of the CDF permits to approximately evaluate $T$ and $T^{-1}$. 
In some specific cases, it can be more convenient to work with an approximate transformation built analytically rather than numerically.

As a starting point, we discuss several examples where the transformations have closed form expressions. Since our transformation is diagonal, our discussion will focus on univariate distributions. The technical derivations of the results are given in the appendix.

\begin{example}[Laplace distribution]\label{ex:Laplace}
The Laplace distribution is often used as a prior because it can be interpreted as $\ell_1$-regularization.
The density of a univariate Laplace distribution is given by $\pi^0(x) = \frac12 \lambda e^{-\lambda |x|}$ for some parameter $\lambda>0$. 
The CDF $\calP^0(x)$ can be analytically computed, which permits us to obtain the formula
$$
 T(z)= -\frac{\text{sign}(z) }{\lambda} \log \big(2\Phi^0( -|z|) \big) ,
\qquad\text{and}\qquad 
 T'(z)=  \frac{\phi^0( -|z|)}{\lambda \Phi^0( -|z|)} ,
$$
for all $z\in \R$, where $\text{sign}(\cdot)$ denotes the sign function with the convention $\text{sign}(0)=0$.
We show in Appendix \ref{sec:Proof_Besov} that the asymptotic behavior of $T$ when  $z\to\pm\infty$ is given by
$$
 T(z) \sim \frac{\text{sign}(z)|z|^2}{2\lambda},
\qquad\text{and}\qquad 
 T'(z) \sim \frac{|z|}{\lambda} .
$$
\end{example}

\begin{example}[Exponential power distribution]\label{ex:Exponential}
The previous example can be generalized to the exponential power distribution with density $\pi^0(x)= \calZ_{p,\lambda}^{-1} e^{-\lambda |x|^p}$ with parameters $p,\lambda>0$ and where $\calZ_{p,\lambda} = \int_{\reals^d}e^{-\lambda |x|^p} \d x$. 
The case $p\geq1$ is typically encountered in the Besov space prior \cite{dashti2012besov,saksman2009discretization}.
We emphasize here that $0<p<1$ is a way to further enforce sparsity, see \cite{hosseini2017well}. Lemma \ref{lem:Besov} shows that the asymptotic behavior of $T(z)$ and $T'(z)$ when $z\rightarrow\pm\infty$ is 
$$
 T(z) \sim \text{sign}(z)\left( \frac{ |z|^{2}}{2\lambda} \right)^{1/p},
\qquad\text{and}\qquad 
 T'(z) \sim \frac{|z|} { \lambda p} \left( \frac{|z|^2}{2 \lambda} \right)^{1/p-1} ,
$$
for any $p>0$.
In particular, small values for $p\ll1$ yield high-order polynomial tails for $T$.
\end{example}

\begin{example}[Cauchy distribution]\label{ex:Cauchy}
 The Cauchy distribution has density $\pi^0(x)= \frac{\lambda}{\pi((\lambda x)^2+1)}$, where $\lambda>0$ is a scale parameter.
 The tails of the Cauchy distribution are so heavy that all moments are undefined, meaning $\int_{\reals^d}|x|^n\d\pi^0=\infty$ for all $n\geq1$.
 Its CDF admits a closed-form expression $\calP^0(x) = \frac{\arctan(\lambda x)}{\pi} +\frac{1}{2}$, which leads to the transformation
 $$
  T(z) = \gamma\tan(\pi\Phi^0(z)-\pi/2).
 $$
 Thus, the asymptotic behavior of $T(z)$ and $T'(z)$ when $z\rightarrow\pm\infty$ is given by
 $$
  T(z) 
  \sim \frac{\gamma z e^{z^2/2}}{\sqrt{\pi/2}} 
  \qquad\text{and}\qquad 
    T'(z) 
    \sim   \frac{ \gamma z^2  e^{z^2/2}}{\sqrt{\pi/2} } .
 $$
\end{example}

\begin{example}[Power-law distributions]\label{ex:Power-law}

Pareto distributions are power-law probability distributions that are often used to model heavy tail phenomenon. The density of the zero symmetric Pareto distribution is defined on $\R$ by $\pi^0(x) = \frac{\alpha}{2}(1+|x|)^{-(\alpha+1)}$ for some parameter $\alpha>0$, and its CDF is $\calP^0(x) = 1-\frac12(1+|x|)^{-\alpha}$ for $x\geq0$ and $\calP^0(x) = \frac12(1+|x|)^{-\alpha}$ for $x\leq0$.
Thus we obtain the formula
$$
 T(z) = - \text{sign}(z) \left( 1-\big(2\Phi^0(-|z|)\big)^{-1/\alpha} \right),
\qquad\text{and}\qquad 
 T'(z) = \frac{2\phi^0(z)}{\alpha (2\Phi^0(-|z|))^{\frac{\alpha+1}{\alpha}}},
$$
for all $z\in\R$.
As shown in Appendix \ref{sec:Proof_Besov}, we have $\Phi^0(-|z|)\sim \tfrac{e^{-z^2/2}}{|z|\sqrt{2\pi}}$ for $z\rightarrow \pm\infty$ so we deduce the following asymptotic behavior
\begin{equation}\label{eq:asymptoticT_Pareto}
T(z)\sim \text{sign}(z) \left(2|z|\sqrt{2\pi}\right)^{\frac{1}{\alpha}}e^{z^2/(2\alpha) },
\qquad\text{and}\qquad 
T'(z)\sim 
\frac{(\pi/2)^{1/(2\alpha)}}{\alpha}
|z|^{1+\frac{1}{\alpha}}e^{z^2/(2\alpha)}. 
\end{equation}
Notice that with $\alpha=1$ we obtain the same asymptotic behaviour of the Cauchy distribution.
With $\alpha\ll1$, the tails of $T$ are even heavier.
In the same way, the Student's t-distribution with density $\pi^0(x)\propto (1+x^2/\alpha)^{-\frac{\alpha+1}{2}}$ is associated with a transformation $T$ which has similar asymptotic behaviour as in \eqref{eq:asymptoticT_Pareto}.
\end{example}

\begin{example}[Horseshoe]\label{ex:Horseshoe}
Some heavy-tailed distributions can be defined through a hierarchical model. For example, the Horseshoe distribution \cite{bhadra2019lasso,carvalho2009handling} considers a latent variable $\gamma$ that follows a half Cauchy distribution on $\R_{\geq0}$ so that, conditioned on $\gamma$, we let $x|\gamma \sim \mathcal{N}(0, (\gamma \tau)^2 )$ follow a Gaussian distribution with standard deviation $\gamma \tau$, where $\tau > 0$ is a global shrinkage parameter. The density function and CDF of the Horseshoe prior are given by 
\[
\pi^0(x) = \int_0^{\infty} \frac{1}{\tau\gamma} \phi^0\left(\frac{x}{\tau\gamma}\right)\frac{2}{\pi(1+\gamma^2)}\d\gamma \quad \text{and} \quad 
\calP^0(x) = \int_0^{\infty} \Phi^0\left(\frac{x}{\tau\gamma}\right)\frac{2}{\pi(1+\gamma^2)}\d\gamma  ,
\]
respectively. Since there is not closed-form expression for these univariate integrals, either numerical quadrature or hierarchical sampling methods need to be used to handle the Horseshoe prior in practice.
As shown in \cite[Eq.(3)]{bhadra2019lasso}, the bound $\tfrac{1}{\tau(2\pi)^{3/2}}\log(1+\tfrac{4\tau^2}{x^2})\leq \pi^0(x)\leq \tfrac{2}{\tau(2\pi)^{3/2}}\log(1+\tfrac{2\tau^2}{x^2})$ holds and yields $\pi^0(x)\sim C (1+|x|)^{-2}$ with $C=\tfrac{4\tau}{(2\pi)^{3/2}}$. Thus, the tails of the Horseshoe density are the same as the one of the power-law density with $\alpha=1$ and the tails of $T(z)$ in $z\rightarrow\pm\infty$ are 
\[
T(z)\sim \text{sign}(z) \left(2|z|\sqrt{2\pi}\right)e^{z^2/2 },
\qquad\text{and}\qquad 
T'(z)\sim 
(\pi/2)^{1/2}
|z|^{2}e^{z^2/2}. 
\]
\end{example}

\subsection{General priors}

When the prior $\pi^0$ is not of a product-form, the transformation $T$ cannot be diagonal. Nonetheless, there exist many ways to  define a (nondiagonal) transformation $T$ such that $T_\sharp\phi^0 = \pi^0$.
A constructive way to build such a $T$ is the Rosenblatt transformation \cite{bogachev2005triangular,baptista2020adaptive,cui2021conditional}. 
It builds on the factorization $\pi^0(x)=\pi^0(x_1)\pi^0(x_2|x_1),\hdots \pi^0(x_d|x_1,\hdots,x_{d-1})$  to construct the $k$-th map component $T_k$ that pushes forward the one-dimensional reference $\phi^0_k(z_k)$ to the conditional marginal density $\pi^0(x_k|x_1,\ldots,x_{k-1})$. The resulting transformation is lower triangular, meaning that the $k$-th component of the map only depends on the first $k$ components of the variable $z$, i.e., $T_k(z) = T_k(z_1,\hdots,z_k)$.
The Rosenblatt transformation can take advantage of the Markov structure that the prior may have to enforce sparsity in the map $T$, which enables fast evaluation of the transformation, see \cite{baptista2021learning,spantini2018inference} for further details.

Optimal transport \cite{villani2009optimal} is another way to define $T$ such that $T_\sharp\phi^0 = \pi^0$. The basic idea is to let $T$ be the transformation which minimizes some \emph{transportation cost} under the constraint $T_\sharp\phi^0 = \pi^0$. 
Optimal transports are not lower triangular in general.
A good survey of the related computational tools can be found in \cite{peyre2019computational}. 
In practice, optimal transports are numerically difficult to construct, especially in high dimensions.

In some situations, we may only have approximate transformations that map the reference Gaussian density to some approximations of the prior densities. For example, the Nataf transformation \cite{lebrun2009innovating,nataf1962determination} defines a map $\That(z)=T_\text{diag}(R z)$, where $T_\text{diag}$ is of a diagonal transform and $R\in\R^{d\times d}$ is a matrix such that $\That_\sharp\phi^0$ has the same marginals and the same covariance as $\pi^0(x)$. 
Examples also include normalizing flows \cite{baptista2020adaptive,rezende2015variational} and generative adversarial networks \cite{goodfellow2014generative}, where the prior density $\pi^0(x)$ needs to be estimated from the samples. 
This underlines the necessity of analyzing the stability of the proposed method when the transformation $T$ approximately pushes forward $\phi^0 $ to $ \pi^0$.

In the subsequent development, we assume that the prior admits a Lebesgue density which is fully supported. This ensures the existence of a map $T$ which is differentiable \cite{bogachev2005triangular}. While these are  reasonable assumptions for most applications, there could be other challenging distributions where $T$ does not exist or is not smooth. This may happen if the support of $\pi^0$ is disjoint or if it does not have a Lebesgue density.

\section{Accelerated MCMC in the reference coordinates}
\label{sec:MCMC} 
Next, we discuss how to combine prior normalization and MCMC algorithms to solve high-dimensional Bayesian inverse problems with heavy-tailed priors. We will present the construction of LIS using prior normalization, followed by LIS-accelerated MCMC sampling. The analysis of the resulting algorithms is provided later in Section \ref{sec:theory}.

\subsection{LIS using prior normalization}

The efficiency of an MCMC algorithm critically depends on the ability of a proposal density to explore the invariant distribution. The high-dimensional parameters of inverse problems make it difficult to design proposals that can tightly follow the geometry of the posterior. We aim to exploit the intrinsic low-dimensional structure of inverse problems to mitigate this challenge. 

The key intuition is that because of the smoothness of the forward model, the incomplete nature of the observations and the noise in the measurement process, the observed data may only inform a subspace of the high-dimensional parameter space. Suppose such a subspace is given by the image of a matrix $U_r\in\R^{d\times r}$ with orthonormal columns, i.e., $U_r^TU_r=I_r$. Let the complement of $\mathrm{Im}(U_r)$ be the image of another matrix $U_\bot\in\R^{d\times (d-r)}$ with orthonormal columns such that $U_r^TU_\bot = 0$. Then, we can decompose the high-dimensional parameter $x$ as
\begin{equation}\label{eq:decomp_x_reference_space}
 x=U_r x_r+ U_\bot x_\bot,
 \quad  \text{where}\quad \left\{
 \begin{array}{l}
  x_r = U_r^T x \\
  x_\bot = U_\bot^T x \\
 \end{array}\right. .
\end{equation}
We denote the marginal posterior and the conditional posterior by $\pi^y(x_r)$ and $\pi^y(x_\perp|x_r)$, respectively. Since the data is only informative to the $r$-dimensional parameter $x_r$, the conditional posterior $\pi^y(x_\perp|x_r)$ can be approximated by the conditional prior $\pi^0(x_\perp|x_r)$. Thus, we can approximate the full posterior $\pi^y(x) = \pi^y(x_r)  \pi^y(x_\perp|x_r)$ by
\begin{equation}\label{eq:ApproximatePosterior}
 \widetilde\pi^y(x) = \pi^y(x_r)  \pi^0(x_\perp|x_r).
\end{equation}
In other words, the likelihood function is effectively supported on $\mathrm{Im}(U_r)$, and thus $\mathrm{Im}(U_r)$ is also referred to as the \emph{Likelihood Informed Subspace}.

The structure suggested in \eqref{eq:ApproximatePosterior} provides a guideline to accelerate MCMC sampling that is analogous to the Rao-Blackwellization principle \cite{robert2021rao}. One should implement state-of-the-art MCMC algorithms, e.g., those inspired by Hamiltonian and Langevin dynamics \cite{bui2014solving,girolami2011riemann,hoffman2014no,martin2012stochastic,neal2011mcmc,petra2014computational,roberts1996exponential}, to target only the marginal posterior $\pi^y(x_r)$, while using the conditional prior $\pi^0(x_\perp|x_r)$ to explore the complement of the LIS. This strategy has been previously investigated in \cite{Cuietal16b,cui2014likelihood} for exploring problems with Gaussian priors. For problems with heavy-tailed prior distributions, there are two major obstacles in implementing this strategy. Firstly, for non-Gaussian prior distributions, the conditional prior $\pi^0(x_\perp|x_r)$ may not be analytically tractable for arbitrary basis $U_r$. Secondly, as shown in \cite{zahm2018certified,cui2020data,cui2021unified}, constructing a suitable $U_r$ with theoretical guarantees generally requires $\pi^0$ to be log-concave, which is not the case for some heavy-tailed priors.

The prior normalization technique allows us to tackle the above-mentioned obstacles. Given the transformation $T$ such that $T^\sharp\pi^0 (z)=\phi^0(z)$, the pullback posterior density $\phi^y(z) = T^\sharp \pi^y (z)$ can be written as
\begin{equation}
    \phi^y(z)  = \frac{1}{\mathcal{Z}}g(z;y)\phi^0(z), \qquad g(z;y)=f(T(z);y), \label{eq:normal_post}
\end{equation}
and be interpreted as a posterior density equipped with a log-concave Gaussian prior. Thus, after prior normalization, the goal of LIS-based posterior approximation becomes finding $U_r$ such that the reference parameter $z$ yields a decomposition
\begin{equation}\label{eq:decomp_z}
 z=U_r z_r+ U_\bot z_\bot,
 \quad  \text{where }\quad \left\{
 \begin{array}{l}
  z_r = U_r^T z \\
  z_\bot = U_\bot^T z \\
 \end{array}\right. ,
\end{equation}
that can accurately approximate the posterior distribution $\phi^y(z)$ in the reference coordinate by $\widetilde\phi^y(z) = \phi^y(z_r)  \phi^0(z_\bot|z_r)$.
Because the prior $\phi^0$ is a standard Gaussian, the conditional prior $\phi^0(z_\bot|z_r)$ is equivalent to the marginal prior $\phi^0(z_\bot)$ for any orthogonal basis $U_r$. Therefore, the marginal posterior  $\phi^y(z_r)=\int_{\reals^{d-r}}\phi^y( U_rz_r + U_\bot^T z_\bot)  \d z_\bot $ can be written as
\begin{equation}\label{eq:ApproximatePosterior_z}
 \phi^y(z_r) =\frac{1}{\calZ}  \gbar( z_r ; y)  \phi^0( z_r ),
 \quad\text{where}\quad 
 \gbar( z_r ; y) = \int_{\reals^{d-r}}g(U_r z_r + U_\bot z_\bot;y)\phi^0(z_\bot)  \d z_\bot.
\end{equation}
As explained later in Section \ref{sec:theory}, one way to construct such a basis $U_r$ is to use the leading eigenvectors (\emph{i.e.} the ones associated with the largest eigenvalues) of the matrix
$$
 H = \int_{\reals^d}\nabla_z \log g(z;y)\nabla_z \log g(z;y)^T \d \phi^y(z).
$$
Note that the matrix $U_\perp$ is introduced for defining the complement of $\mathrm{Im}(U_r)$ and, in practice, we do not need to assemble it explicitly. For instance, the projection onto $\mathrm{Im}(U_\bot)$ can be obtained using the orthogonal projector $I - U_r U_r^T$. In Section \ref{sec:41}, we analyze the accuracy of the approximate posterior induced by $U_r$ using the above matrix $H$. In the rest of this section, we will focus on how to accelerate MCMC sampling supposing the basis $U_r$ is given. 

Given a LIS basis $U_r$, pushing forward $\widetilde\phi^y(z)$ through the transformation $T$, the resulting approximate posterior density in the original coordinate can be expressed as
\begin{equation}
    \label{eq:ApproximatePosterior_Tz} 
    T^\sharp \widetilde \phi^y(x) = \gbar\big( U^T_r T^{-1}(x) ; y\big) \pi^0(x) .
\end{equation}
Although $T^\sharp \widetilde \phi^y$ does not follow the form of the reduced-dimensional posterior $\widetilde \pi^y$ in \eqref{eq:ApproximatePosterior}, it can be interpreted as a nonlinear reduced-dimensional approximation to the likelihood, see \cite{bigoni2021nonlinear}. Indeed, the approximate likelihood function $x\mapsto \gbar( U^T_r T^{-1}(x) ; y )$ is constant on the $d-r$ dimensional manifolds 
$$
\mathcal{M}_{x_0}=\left\{T( U_rU_r^T T^{-1}(x_0) + U_\bot z_\bot ) , z_\bot\in\R^{d-r}\right\} ,
$$ 
for any $x_0\in\R^d$.
In other words, instead of the decomposition \eqref{eq:decomp_x_reference_space}, the dimension reduction in the original space yields a decomposition
$$
 x = T \big( U_r z_r + U_\bot z_\bot  \big),
 \quad \text{where }\quad \left\{
 \begin{array}{l}
  z_r = U_r^T T^{-1}(x) \\
  z_\bot = U_\bot^T T^{-1}(x) \\
 \end{array}\right. ,
$$
and $z_r$ contains the informed coordinates and $z_\bot$ represents the non-informed ones. 

\subsection{Exact inference using pseudo-marginal}

The approximate posterior \eqref{eq:ApproximatePosterior_Tz} can be explored using a decomposed strategy. One can first $m$ draw samples from the low-dimension marginal posterior $z_r^i\sim \phi^y(z_r), i=1,\ldots,m,$ using an MCMC algorithm, and then draw samples from the standard normal $z_\bot^i\sim \phi^0(z_\bot)$ to obtain approximate posterior samples $x_i = T ( U_r z_r^i + U_\bot z_\bot^i  )$. However, the critical issue in implementing this strategy is that we need to evaluate the marginal likelihood function $\gbar(z_r;y)=\int_{\reals^{d-r}}g(U_r z_r + U_\bot z_\bot;y)\phi^0(z_\bot)  \d z_\bot $, which involves an high-dimensional integral with no analytical solution available in general.
Fortunately, we can employ the pseudo-marginal method \cite{andrieu2009pseudo} to overcome this limitation by defining auxiliary MCMC transition kernels that tightly follow the structure of the approximate posterior. This naturally extends standard MCMC algorithms that are efficient for low or moderate dimensional problems to simultaneously explore the marginal posterior and the original full posterior.

Considering the posterior $\phi^y(z_r, z_\bot)$ after the prior normalization, to construct MCMC transition kernels following the principle of pseudo-marginal, we first extend the complementary subspace to define the auxiliary posterior density
\begin{equation}
\phi^{y,m}\left(z_r, \{z_\bot^i\}_{i = 1}^m\right) = \frac1{\mathcal{Z}} \phi^0(z_r)  \bigg(\prod_{i=1}^m \phi^0(z_\bot^i) \bigg) \bigg(\frac1m \sum_{i=1}^m g(z_r, z_\bot^i;y)\bigg). \label{eq:aux1}
\end{equation}
Since the conditional posterior $\phi^y(z_\bot|z_r)$ can be expressed as 
\begin{equation}
 \label{tmp:264589}
\phi^y(z_\bot|z_r) = \frac{\phi^y(z_r,z_\bot)}{\phi^y(z_r)} = \frac{\phi^0(z_\bot) g(z_r, z_\bot;y)}{\gbar(z_r; y)},
\end{equation}
the auxiliary posterior density can also be written as
\begin{align}
\phi^{y,m}\left(z_r, \{z_\bot^i\}_{i = 1}^m\right) & = \frac1{\mathcal{Z}} \phi^0(z_r) \gbar(z_r; y) \bigg(\frac1m \sum_{i=1}^m \frac{g(z_r, z_\bot^i;y)}{\gbar(z_r; y)} \prod_{j=1}^m \phi^0(z_\bot^j) \bigg) \nonumber \\
& = \phi^y(z_r) \bigg(\frac1m \sum_{i=1}^m \phi^y(z_\bot^i|z_r) \prod_{j\neq i} \phi^0(z_\bot^j) \bigg), \label{eq:aux2}
\end{align}
where we applied \eqref{tmp:264589} for $\phi^y(z^i_\bot|z_r)$.
Using the above identity, marginalizing $\phi^{y,m}(z_r, \{z_\bot^i\}_{i = 1}^m)$ over all the complementary coordinates $z_\bot^1,\ldots, z_\bot^m$, we obtain the marginal posterior $\phi^{y}(z_r)$. Therefore, constructing a Markov chain transition kernel that is invariant to the auxiliary density $\phi^{y,m}(z_r, \{z_\bot^i\}_{i = 1}^m)$ also leads to marginal Markov chains that can sample the marginal posterior $\phi^{y}(z_r)$.

The key property we exploit to design efficient MCMC algorithms is that the conditional posterior $\phi^y(z_\bot|z_r)$ can be approximated by the marginal prior on the complement of the LIS, i.e., $\phi^y(z_\bot|z_r) \approx \phi^0(z_\bot)$. Thus, the term in the brackets of \eqref{eq:aux2} can be also approximated by the product of marginal priors, i.e.,
\[
    \frac1m \sum_{i=1}^m \phi^y(z_\bot^i|z_r) \prod_{j\neq i} \phi^0(z_\bot^j) \approx  \prod_{j = 1}^m \phi^0(z_\bot^j).
\]
This way, we can define decomposed MCMC proposal densities in the form of 
\begin{equation}
    p\left( z_r^\prime, \{z_\bot^{\prime i}\}_{i = 1}^m | z_r, \{z_\bot^i\}_{i = 1}^m \right) = p( z_r^\prime | z_r ) \prod_{i =1}^m \phi^0(z_\bot^i),\label{eq:aux_prop}
\end{equation}
where we want the low-dimensional LIS proposal $ p( z_r^\prime | z_r )$ to follow the structure of the marginal posterior, while using complementary prior to explore the rest. Here we aim to implement the LIS proposal using well-established MCMC algorithms. For example, the automatically tuned no-u-turn sampler (NUTS) of \cite{hoffman2014no}, adaptive Metropolis adjusted Langevin algorithm (MALA) of \cite{atchade2006adaptive}, the pCN proposal of \cite{beskos2008mcmc,cotter2013mcmc}, transport-map samplers \cite{parno2018transport}, etc. 

To sample the invariant density $\phi^{y,m}(z_r, \{z_\bot^i\}_{i = 1}^m)$, we need to accept proposed samples of the auxiliary proposal \eqref{eq:aux_prop} with probability 
\begin{align}
    \alpha_1\left(  z_r, \{z_\bot^i\}_{i = 1}^m; z_r^\prime, \{z_\bot^{\prime i}\}_{i = 1}^m \right) & = 1 \wedge \frac{\phi^{y,m}(z_r^\prime, \{z_\bot^{\prime i}\}_{i = 1}^m)}{\phi^{y,m}(z_r, \{z_\bot^i\}_{i = 1}^m)} \frac{p( z_r, \{z_\bot^i\}_{i = 1}^m | z_r^\prime, \{z_\bot^{\prime i}\}_{i = 1}^m )}{p( z_r^\prime, \{z_\bot^{\prime i}\}_{i = 1}^m | z_r, \{z_\bot^i\}_{i = 1}^m )} \nonumber \\
    & = 1 \wedge \frac{\phi^0(z_r^\prime) \sum_{i=1}^m g(z_r^\prime, z_\bot^{\prime i};y)}{\phi^0(z_r) \sum_{i=1}^m g(z_r, z_\bot^i;y)} \frac{p( z_r | z_r^\prime )}{p( z_r^\prime | z_r )}
\end{align}
where the second equation above follows from \eqref{eq:aux1} and \eqref{eq:aux_prop}, and we use $a\wedge b$ to denote  $\min\{a,b\}$. Given the above acceptance probability, the MCMC transition kernel originally defined for the auxiliary posterior density \eqref{eq:aux1} can  be interpreted as a lower-dimensional MCMC transition kernel that makes acceptance/rejection on the Monte Carlo average  
\[
 \phi^0(z_r) \bigg( \frac1m \sum_{i=1}^m g(z_r, z_\bot^i;y)\bigg) , \quad  z_\bot^i \sim \phi^0(z_\bot),
\]
which is an unbiased estimator of the unnormalized marginal posterior density. This way, one natural question to ask is how the LIS basis $U_r$ impact the efficiency of the pseudo-marginal method. In Section \ref{sec:pseudoanalysis}, we address this question by combining the error estimates of the approximate posterior defined by $U_r$ and Corollary 4 of \cite{andrieu2015convergence}. 

The above pseudo-marginal method can also generate samples from the full posterior while sampling the marginal posterior.  For a given $z_r$ the evaluated likelihood functions yield a set of weighted complementary prior samples $\{z_\bot^i, w_\bot^i\}_{i = 1}^m$, where $w_\bot^i = g(z_r, z_\bot^i;y)$, that can be viewed as weighted samples of the conditional posterior $\phi^y(z_\bot|z_r)$. This way, for a state of the auxiliary Markov chain, $(z_r, \{z_\bot^i\}_{i = 1}^m)$, we can randomly select a complementary sample $z_\bot^\ast$ from the set $\{z_\bot^i\}_{i =1}^m$ according to the categorical distribution defined by the unnormalized weights $\{w_\bot^i\}_{i =1}^m$ and assemble a full posterior sample by $z = U_r z_r + U_\bot z_\bot^\ast$. 
Using this strategy, we can simultaneously sample the marginal posterior and the full posterior. One step of the resulting MCMC algorithm is given in Algorithm \ref{alg:pMCMC}.

\begin{algorithm}[h]  
 \SetAlgoLined
 \SetKwFunction{algo}{pseudo-marginal}
 \SetKwFunction{proc}{propose}
 \SetKwFunction{avrg}{average}
\SetKwInOut{Required}{Required}
\Required{transformation $T$, matrix $U_r$, likelihood $f(x;y)$, LIS proposal $p(\cdot|\cdot)$ on $\R^r$, and a pseudo-marginal sample size $m$.}
\KwIn{current state $x$ and the associated Monte Carlo average $R$.}
\KwOut{new state $x'$ and the the new Monte average $R'$.}

\SetKwProg{myalg}{Algorithm}{}{}
\SetKwProg{myproc}{Procedure}{}{}

\BlankLine
\myalg{\algo{$x$, $R$, $m$}}{
        $(x^\prime, R^\prime, \alpha_1) \leftarrow$ \texttt{propose}$(x, R, T, U_r, f, p, m)$\;        
        \If(reject the proposal candidate){$\mathrm{uniform}[0,1] > \alpha_1$}{Set $x' = x$ and $R' = R$\;}
        \KwRet $x^\prime$ and $R^\prime$\;
}

\BlankLine
\myproc{\proc{$x, R, T, U_r, f, p, m$}}{
     Evaluate the reference parameter $z=T^{-1}(x)$ and compute $z_r = U_r^T z$\;
     Generate a LIS proposal candidate $z'_r\sim p(\,\cdot\,|z_r)$\;
    \For{$i = 1, \ldots, m$}{
         Generate reference prior samples $z^i \sim \mathcal{N}(0, I)$\; 
         Project $z^i$ to the complement of LIS $z^i = z^i - U_r (U_r^T z^i)$\; 
         Set $x^i = T(U_r z'_r + z^i)$ and compute the likelihood weights $w^i_\bot = f(x^i ;y)$\;
    }
     Compute the Monte Carlo average
    \(
         R' = \phi^0(z_r) \big(  \frac{1}{m}\sum_{i=1}^m w^i_\bot \big)
    \)\;
     Compute the acceptance probability
        \(
           \alpha_1 = 1 \wedge ( R' \, p(z_r|z'_r)) \big/ (R \, p( z'_r|z_r))
        \)\;
     Draw a sample $x^\prime$ from $\{x^{i}\}_{i=1}^m$ according to the weights  $\{w^{i}_\bot\}_{i=1}^m$\;
     \KwRet $x^\prime$, $R^\prime$ and $\alpha_1$\;
}
\caption{One step of LIS-pseudo-marginal MCMC with prior normalization}
\label{alg:pMCMC}
\end{algorithm}

\subsection{Delayed acceptance for approximate prior normalization}
As discussed in Section \ref{sec:Tconstruct}, we may have access only to an approximate transformation $\That$ rather than the exact $T$ for the prior normalization. In this situation, we can still apply Algorithm \ref{alg:pMCMC} to generate samples from an approximate posterior using the approximate transformation. Then, assuming that we can evaluate the original prior density $\pi^0(x)$, we can remove the approximation error by applying the delayed acceptance method \cite{christen2005markov,liu1998sequential}.

As a starting point, we use the approximate transformation $\That$ and the likelihood function $x\mapsto f(x;y)$ to define an approximate posterior density in the form of
\[
    \phihat^y(z)  = \frac{1}{\widehat{\mathcal{Z}}} \widehat{g}(z;y)\phi^0(z), \qquad \widehat{g}(z;y)=f(\That(z);y).
\]
The pushforward density of $\phihat^y(z)$ under the transformation $\That$, which is given as
\begin{equation}
    \pihat^y(x) = \frac{1}{\widehat{\mathcal{Z}}} f(x;y)\pihat^0(x), \qquad \pihat^0(x) = \mathrm{det}(\nabla \That(\That^{-1}(x)) )\phi^0(\That^{-1}(x)),\label{eq:da_approx}
\end{equation}
defines an approximation to the original posterior $\pi^y(x)\propto f(x;y)\pi^0(x)$. We can apply Algorithm \ref{alg:pMCMC} to sample the approximate density $\pihat^y(x)$, as the associated reference density $\phihat^y(z)$ follows a similar structure to that in \eqref{eq:normal_post}. For situations where the approximate prior $\pihat^0$ is close to the original prior $\pi^0$, the delay acceptance methods can have good efficiency in removing the approximation error caused by $\That$. The detail of the delayed acceptance method is given by Algorithm \ref{alg:xMCMC}.

\begin{algorithm}[H]
    \SetAlgoLined
    \SetKwFunction{algo}{delayed-acceptance}
   \SetKwInOut{Required}{Required}
   \Required{approximate transformation $\That$, matrix $U_r$, likelihood $f(x;y)$, prior density $\pi^0(x)$, LIS proposal $p(\cdot|\cdot)$ on $\R^r$, and a pseudo-marginal sample size $m$.}
   \KwIn{current state $x$ and the associated Monte Carlo average $R$.}
   \KwOut{new state $x'$ and the the new Monte average $R'$.}
   
   \SetKwProg{myalg}{Algorithm}{}{}
   \SetKwProg{myproc}{Procedure}{}{}
   
   \BlankLine
   \myalg{\algo{$x$, $R$, $m$}}{
            $(x^\prime, R^\prime, \alpha_1) \leftarrow$ \texttt{propose}$(x, R, \That, U_r, f, p, m)$\;        
           \eIf(delay the acceptance){$\mathrm{uniform}[0,1] > \alpha_1$}{ Set $x' = x$ and $R' = R$\;}{
            Compute the acceptance probability 
           \[
           \alpha_2(x;x')=
               1\wedge \frac{\pi^0(x') \text{det}(\nabla \That(z)) \phi^0(z)}{\pi^0(x)\text{det}(\nabla \That(z'))\phi^0(z')},
           \]
           where $z = \That^{-1}(x)$ and $z' = \That^{-1}(x')$\;
           \If(reject the proposal candidate){$\mathrm{uniform}[0,1] > \alpha_2$}{ Set $x' = x$ and $R' = R$\;}
           }
            \KwRet $x^\prime$ and $R^\prime$\;
   }
   \caption{One step of the delayed acceptance MCMC for approximate prior normalization}
   \label{alg:xMCMC}
\end{algorithm}

Here, the procedure \texttt{propose} is the same as that of Algorithm \ref{alg:pMCMC}. A surprising fact of Algorithm \ref{alg:xMCMC} is that it does not need the information of the exact transformation $T$ to sample the original posterior. Since the approximation error was caused by only the approximate prior normalization, the second step acceptance probability in Line 9 of Algorithm \ref{alg:xMCMC} only uses the original prior density and the approximate transformation to correct the error. The performance analysis of Algorithm \ref{alg:xMCMC} is given in Section \ref{sec:daanalysis}. 

\section{Approximation analysis with transformation}
\label{sec:theory}
In this section, we explain how to find the LIS. We also provide rigorous bounds on the associated approximations, and show their implication for MCMC efficiency. 

\subsection{Gradient-based construction of $U_r$}
\label{sec:41}
We review here the gradient-based method \cite{zahm2018certified,cui2021unified} to construct the matrices $U_r$ and $U_\bot$.
Here, we work in the reference coordinate $z$. We recall that the posterior is $\phi^y(z) \propto g(z;y) \phi^0(z)$ where $\phi^0(z)$ is the standard normal density and where $g(z;y)=f(T(z);y)$ is the likelihood function.
The goal is to construct a matrix $U_r\in\R^{d\times r}$ with orthogonal columns such that
$$
 \widetilde \phi^y(z) =  \phi^y(z_r) \phi^0(z_\bot|z_r) ,
$$
is a good posterior approximation, where $z_r = U_r^T z$ and $z_\bot = U_\bot^T z$. Here, $U_\bot\in\R^{d\times (d-r)}$ is any matrix with orthogonal columns such that $U_r^TU_\bot=0$.
As shown in Corollary 2.10 of \cite{zahm2018certified}, because $\phi^0(z)$ is the standard normal density we have that the Kullback-Leibler divergence $\Dkl( \phi^y || \widetilde \phi^y ) = \int_{\reals^d}\log(\phi^y(z) /\widetilde \phi^y(z) ) \pi^y (z) \d z $ can be bounded by
\begin{equation}\label{eq:KLbound}
    \Dkl( \phi^y || \widetilde \phi^y )\leq \frac{1}{2}
    \res(U_r , H) ,
\end{equation}
where
$$
 H = \int_{\reals^d}\nabla_z \log g(z;y)\nabla_z \log g(z;y)^T  \phi^y(z) \d z,
$$
and where 
\begin{equation}\label{eq:residual}
    \res(U_r , H) = 
    \trace(H) - \trace(U_r^T H U_r),
\end{equation}
is the \emph{trace residual} of $H$ on the subspace spanned by $U_\bot$.
Theorem 2.4 of \cite{cui2021unified} establishes a similar bound on the Hellinger distance $\Hell( \phi^y , \widetilde \phi^y ) = \big( \frac12\int_{\reals^d}((\phi^y(z))^{1/2} - (\widetilde \phi^y(z))^{1/2} )^2 \d z\big)^{1/2} $ as follow
$$
 \Hell( \pi^y , \widetilde \pi^y )^2 \leq \frac{1}{4}\res(U_r , H).
$$
The matrix $U_r$ can be constructed by minimizing the error bound. As shown in  \cite{zahm2020gradient,kokiopoulou2011trace}, the minimum of $\res(U_r , H)$ is attained with the matrix $U_r$ which contains the $r$ largest eigenvectors of $H$. 
Denoting by $u_i$ the $i$-th largest eigenvector of $H$, i.e. $H u_i = \lambda_i u_i$, we define $U_r$ and $U_\bot$ as
\begin{align*}
    U_r &= [u_1,\hdots, u_r]   \\
    U_\bot &= [u_{r+1},\hdots, u_d]  .
\end{align*}
With this optimal choice, we obtain the bounds
\begin{align*}
    \Dkl( \phi^y || \widetilde \phi^y ) &\leq \frac{1}{2}( \lambda_{r+1}+\hdots+\lambda_d), \\
    \Hell( \pi^y , \widetilde \pi^y )^2 &\leq \frac{1}{4}( \lambda_{r+1}+\hdots+\lambda_d),
\end{align*}
which relates the errors on the posterior density with the spectrum of $H$.

\subsubsection{Variable selection in the reference coordinate}
\label{sec:varsele}
Since $T$ is a nonlinear transformation, the informed subspace $\text{Im}(U_r)$ in the reference coordinate is mapped into a \textit{curved manifold} in the original coordinates, which may lack interpretability. This issue can be resolved if $T$ is a diagonal transform (\emph{e.g.} if $\pi^0$ has a product form) and if $U_r$ is a coordinate selection matrix, that is, $U_r^T z = (z_{\tau_1},\hdots,z_{\tau_r})$ 
for some set of indices $\tau\subset\{1,\hdots,d\}$. It is easy to check that
$$
 T( U_rz_r+U_\bot z_\bot )
 = U_r T_\tau(z_r) + U_\bot T_{-\tau}(z_\bot) ,
$$
where $T_{\tau} = (T_{\tau_1},\hdots,T_{\tau_r})$ and $-\tau = \{1,\hdots,d\} \backslash \tau$.
In other words, the informed subspace $\text{Im}(U_r)$ is invariant under $T$.
The choice of the informed indices $\tau$ can also be made easily. By definition \eqref{eq:residual} we have
\[
\res(U_r , H) = \trace(H) - \sum_{i=1}^r H_{\tau_i,\tau_i} ,
\]
where $H_{\tau_i,\tau_i}$ is the $\tau_i$-th diagonal term of $H$. 
In order to minimize $\res(U_r , H)$, we want $H_{\tau_1,\tau_1},\hdots, H_{\tau_r,\tau_r}$ to be as large as possible. To do this, $\tau$ needs to contain the $r$ largest diagonal terms of $H$. This coordinate selection procedure is, in principle, very similar to the screening methods in sensitivity analysis \cite{saltelli2008global}, where the goal is to identify the parameter components (\emph{e.g.} the coordinates) which are the most relevant for explaining the variability of given model responses.

\subsubsection{Well-definedness of the matrix $H$}

We address here the important question of the well-defined of the matrix $H$.
Note that $H$ is well defined is equivalent to its trace being finite, which is checking whether
\begin{align*}
    \trace(H) 
 &= \int_{\reals^d}\|  \nabla_z \log g( z;y ) \|^2 \d\phi^y(z) \\
 &= \int_{\reals^d}\| \nabla_z T(z) \nabla_x \log f( T(z);y ) \|^2 \frac{f(T(z);y)}{\calZ} \d\phi^0(z)  ,
\end{align*}
is finite. 

\begin{proposition}\label{prop:Hbound}
Assume that both $z\mapsto T(z)$ and $x\mapsto f(x;y)$ are continuously differentiable. Then the following scenario ensures $\trace( H )<\infty$:
\begin{enumerate}

\item There exists $C>0, \lambda\geq 0$ and $p>0$, such that for all $x\in\R^d$ we have
\begin{align*}
    \|\nabla T(z)\|  &\leq C\exp(\lambda \|z\|),\quad
     \|T(z)\| \leq C\exp(\lambda \|z\|), \\
     f(x; y) &\leq C,\quad
     \|\nabla_x \log f(x; y)\| \leq C(\|x\|+1)^p.
\end{align*}
\item There exists $\alpha>2$ and $C>0$, so that
 $$\|\nabla T(z)\|\leq C\exp(\tfrac1{2\alpha}\|z\|^2),$$
 which ensures $\|\nabla T(z)\|^2$ is integrable. Moreover, either $\max\{\|\nabla \log f\|, f\}\leq C$ or $f(x;y) = C \exp(-\tfrac1{2\sigma^2}\|Ax-y\|^2)$.
 \item The function $x\mapsto \|\nabla \log f (x;y) \|$ has bounded support.
\end{enumerate}
\end{proposition}

Comparing with the examples discussed earlier in Section \ref{sec:diagonalprior}, we can see Scenario (i)-(iii) are considering priors with increasingly heavier tails: scenario (i) holds for exponential power distributions (see examples \ref{ex:Laplace} and \ref{ex:Exponential}); scenario (ii) holds for distribution with polynomial tails (see examples \ref{ex:Cauchy} and \ref{ex:Power-law}), where the index $\alpha$ needs to be  larger than 2; scenario (iii) has no constraints on $\pi^y$, so it can work for Cauchy distribution which has index $\alpha=1$.

On the other hand, with heavier tails, we have more restrictions on the likelihood function. Scenario $1$ holds for general nonlinear inverse problems, of which the likelihood is of form $f(x;y)\propto \exp(-\frac12\|G(x)-y\|^2)$.
Scenario $2$ requires the log-likelihood to be Lipschitz or the inverse problem is linear. Scenario $3$ requires the likelihood to be constant outside some bounded set. While this is a strong requirement, it is reasonable for scenarios where we know the range of the solution.

\begin{proof}
For the first claim, we simply write
\begin{align*}
    \trace(H) 
 &= \frac{1}{\calZ}\int_{\reals^d}\| \nabla_z T(z) \nabla_x \log f( T(z);y ) \|^2 f(T(z);y) \d\phi^0(z) \\
 &\leq \frac{1}{\calZ \sqrt{2\pi}^d}\int_{\reals^d}\Big( C\exp(\lambda \|z\|) \Big)^2
 \Big(C(|C\exp(\lambda \|z\|)|+1)^p\Big)^2 C \exp(-\|z\|^2/2)\d z \\
 &= \frac{C^5}{\calZ \sqrt{2\pi}^d}\int_{\reals^d}
 (C\exp(\lambda \|z\|)+1)^{2p}  \exp( 2\lambda \|z\| -\|z\|^2/2)\d z \\
 &\leq \frac{C^5 (C+1)^{2p}}{\calZ \sqrt{2\pi}^d}\int_{\reals^d}
 \exp( 2(1+p)\lambda \|z\| -\|z\|^2/2)\d z
 < \infty .
\end{align*}

For the second claim,  if $\|\nabla \log f(T(z);y)\|\leq C$, then 
\begin{align*}
\trace(H)&
= \frac{1}{\calZ}\int_{\reals^d}\| \nabla_z T(z) \nabla_x \log f( T(z);y ) \|^2 f(T(z);y) \d\phi^0(z)\\
&\leq \frac{C^2}{\calZ}\int_{\reals^d}\| \nabla_z T(z)\|^2 f(T(z);y) \d\phi^0(z)\\
&\leq \frac{C^3}{\calZ}\int_{\reals^d}\exp(\tfrac1\alpha \|z\|^2) \d\phi^0(z) \d z<\infty. 
\end{align*}
Also, if $f(x;y) = C \exp(-\tfrac1{2\sigma^2}\|Ax-y\|^2)$, then
$\|\nabla \log f(T(z);y)\|\leq \|A\|\|Ax-y\|$. We deduce that
\begin{align*}
\|\nabla \log f(x;y)\|^2f(x;y)&\leq C\exp(-\frac{1}{2\sigma^2}\|Ax-y\|^2) \|A\|^2\|Ax-y\|^2 
\leq C\|A\|^2 \sqrt{2}\sigma ,
\end{align*}
where for the last inequality we used the fact that $t \exp(-\frac{t}{2\sigma^2}) \leq 2\sigma^2$ holds for any $t\geq0$, in particular for $t=\|Ax-y\|^2$. Then we have
\begin{align*}
\trace(H)&= \frac{1}{\calZ}\int_{\reals^d}\| \nabla_z T(z) \nabla_x \log f( T(z);y ) \|^2 f(T(z);y) \d\phi^0(z)\\
&\leq \frac{C\|A\|^2 \sqrt{2}\sigma }{\calZ\sqrt{2\pi}^d}\int_{\reals^d}\exp(\frac{1}{\alpha} \|z\|^2) \exp(-\frac12\|z\|^2) \d z<\infty. 
\end{align*}

The third claim is trivial: if $x\mapsto \|\nabla_x \log f (x;y) \|$ has a bounded support $K$ then $\trace(H)$ is finite as the integral of a continuous function over a bounded domain $K$.
\end{proof}

\subsection{LIS with approximate transformation}

As discussed earlier, in some cases one only has access to an approximation $\That$ of the transformation $T$. Using $\That$ leads to a different posterior in the reference prior 
\[
\phihat^y(z)=\frac{1}{\widehat{\calZ}}f(\That(z);y) \phi^0(z),\quad \widehat{\calZ}=\int_{\reals^d}f(\That(z);y) \phi^0(z) \d z. 
\]
The push-forward prior $\pihat^0=\That_\sharp \phi$ and  push-forward posterior $\pihat^y=\That_\sharp \phihat^y$ are thus approximation to the prior and posterior, respectively. 
The following proposition controls the error between the posteriors with the error in the transformation.

\begin{proposition}
\label{prop:errT}
Suppose $x\mapsto\log f(x;y)$ is $C$-Lipschitz and that with an $\epsilon>0$
\begin{equation}\label{eq:error_T}
\|T(z)-\That(z)\|\leq \varepsilon ,
\end{equation}
holds for any $z$. 
Then $\Hell(\phihat^y, \phi^y)^2\leq 1-\exp(-C\varepsilon)=C\varepsilon + \mathcal{O}(\varepsilon^2)$. 
In addition, if $z\mapsto \log(\det(\nabla T(z)))$ is $C$-Lipschitz and 
\begin{equation}\label{eq:error_T2}
\|T^{-1}(x)-\That^{-1}(x)\|\leq \varepsilon ,
\qquad 
\exp(-C\varepsilon )\leq \frac{\text{det}(\nabla \That(z))}{\text{det}(\nabla T(z))}\leq \exp(C\varepsilon )
\end{equation}
for any $x,z$, then we also have
$\Hell(\pihat^y, \pi^y)^2 = \Omega \varepsilon + \mathcal{O}(\varepsilon^2)$ for some $\Omega>0$ that may depend on $C$.

\end{proposition}

\begin{proof}
Because $x\mapsto\log f(x;y)$ is $C$-Lipschitz and because $\|T(z)-\That(z)\|\leq \varepsilon$ we have 
$$
 \exp(-C\varepsilon)\leq \frac{ f(T(z);y)}{f(\That(z),y)}\leq \exp(C\varepsilon),
$$
for any $z$ so that
\[
 \exp(-C\varepsilon)\calZ\leq \widehat{\calZ}=\int_{\reals^d}f(\That(z);y) \phi^0(z)\d z\leq \exp(C\varepsilon)\calZ. 
\]
Then
\begin{align*}
 \Hell(\phihat^y, \phi^y)^2
 &=1- \int_{\reals^d}\sqrt{\frac{\phihat^y(x)}{\phi^y(x)}} \phi^y(x)\d x\\
 &=1- \sqrt{\frac{\calZ}{\Zhat}} \int_{\reals^d}\sqrt{\frac{f(\That(z);y)}{f(T(z);y)}} ~ \phi^y(x)\d z \\
 &\leq 1-\exp(-C\varepsilon), 
\end{align*}
holds and yields the first claim.
For the second claim, we start with
\begin{align*}
 \Hell(\pihat^y, \pi^y)^2
 =1- \int_{\reals^d}\sqrt{\frac{\pihat^y(x)}{\pi^y(x)}} \pi^y(x)\d x
 =1- \int_{\reals^d}\sqrt{\frac{\pihat^y(T(z))}{\pi^y(T(z))}} \phi^y(z)\d z.
\end{align*}
By definition we have
\begin{align*}
 \pihat^y(x) &=\frac{1}{\Zhat}f(x;y)\det(\nabla \That(\That^{-1}(x)))\phi^0(\That^{-1}(x)) \\
 \pi^y(x)&=\frac{1}{\calZ}f(x;y)\det(\nabla T(T^{-1}(x)))\phi^0(T^{-1}(x)),
\end{align*}
so that
\begin{align*}
 \Hell(\pihat^y, \pi^y)^2
 &=1- \underbrace{\sqrt{\frac{\calZ}{\Zhat}}}_{=A}\int_{\reals^d}\underbrace{\sqrt{\frac{ \det(\nabla \That(\That^{-1}(T(z))))}{ \det(\nabla T(z)) } }}_{=B} \underbrace{\exp\left( - \frac{\|\That^{-1}(T(z))\|^2-\|z\|^2}{4}\right)}_{=D} \phi^y(z)\d z 
\end{align*}
From the previous relation, we have $A\geq \exp(-C\varepsilon/2)$.
To bound $B$, we first notice that $\|T^{-1}(x)-\That^{-1}(x)\|\leq \varepsilon$ for all $x$ implies $\|z-\That^{-1}((T(z))\|\leq \varepsilon$ for all $z$.
Thus, because $z\mapsto\log |\det(\nabla T(z))|$ is $C$-Lipschitz and by \eqref{eq:error_T2} we have
\begin{align*}
 \frac{ |\det(\nabla \That(\That^{-1}(T(z))))|}{| \det(\nabla T(z)) |} 
 &=\frac{| \det(\nabla T(\That^{-1}(T(z))))|}{| \det(\nabla T(z)) |} \frac{| \det(\nabla \That(\That^{-1}(T(z))))|}{| \det(\nabla T(\That^{-1}(T(z)))) |} 
 \geq \exp(-2C\varepsilon) .
\end{align*}
Thus $B\geq \exp(-C\varepsilon)$.
To bound $C$, we write
\begin{align*}
 \|\That^{-1}(T(z))\|^2 - \|z\|^2
 &= \big( \|\That^{-1}(T(z))\| - \|z\| \big)\big( \|\That^{-1}((T(z))\| + \|z\| \big) \\
 &\leq  \|\That^{-1}(T(z)) - z\|  \big( \|\That^{-1}((T(z))-z\| + 2\|z\| \big) \\
 &\leq \varepsilon \big( \varepsilon  + 2\|z\| \big) = \varepsilon^2 +2\varepsilon\|z\|
\end{align*}
Thus $D\geq \exp(-\varepsilon^2/4 )\exp(- \varepsilon\|z\|/2)$.
We deduce that
\begin{align*}
 \Hell(\pihat^y, \pi^y)^2 
 &\leq 1- \exp(-\varepsilon^2/4 - 3C\varepsilon/2) \int_{\reals^d}\exp(- \varepsilon\|z\|/2) \phi^y(z)\d z  \\
 &\leq 1- \exp(-\varepsilon^2/4 - 3C\varepsilon/2) \underbrace{\int_{\reals^d}\exp(- \|z\|/2) \phi^y(z)\d z }_{E}
\end{align*}
where we assumed $\varepsilon\leq1$ for the last inequality. Obviously $E\neq0$, so we deduce $\Hell(\pihat^y, \pi^y)^2  \leq  3CD\varepsilon/2 + \mathcal{O}(\varepsilon^2) = \Omega \varepsilon + \mathcal{O}(\varepsilon^2)$ where $\Omega=3CE/2$.

\end{proof}

The dimension reduction method can also be implemented using $\That$ instead of $T$. In that case, we compute the matrix
\[
\Hhat=\int_{\reals^d}\nabla \That(z)  \nabla \log f(\That(z);y) \nabla \log f(\That(z);y)^T\nabla \That(z)^T \phihat^y (z) \d z ,
\]
and we construct the matrix $\hat U_r$ by minimizing $\res(U_r , \Hhat)$. 
The resulting approximation posterior is
\[
\apdfhat^y(z)=\frac{1}{\widehat{\calZ}}\ghat( \hat U_r^T z ;y) \phi^0(z),\quad \ghat(z_r;y)=\int_{\reals^{d-r}}f(\That( \hat U_r z_r + \hat U_\bot z_\bot ) ; y) \phi^0(z_\bot)\d z_\bot. 
\]
Note that because $\phihat^y$ and $\apdfhat^y$ are $\phi^y$ and $\apdf$ constructed with $T$ replaced by $\That$, we can use Proposition \ref{prop:errT} and triangle inequality to obtain the following corollary:
\begin{corollary}
The following bounds hold for the approximated posteriors constructed using $\That$:
\[
\Hell(\phihat^y, \apdfhat^y)^2\leq \frac14 \res(U_r , \Hhat),
\quad \Hell(\phi^y, \apdfhat^y)^2\leq \frac14  \res(U_r , \Hhat)+\mathcal{O}(\varepsilon).
\]
\end{corollary}
In other words, using $\apdfhat^y$ is a good approximation of the accurate posterior $\phi^y$ if the residual $\res(U_r , \Hhat)$ is small and if $\That$ is a good approximation of $T$.

\subsection{MCMC efficiency analysis}
\label{sec:MCMCaccept}
Next, we reveal how the LIS analysis connects with the acceptance rate of Algorithms \ref{alg:pMCMC} and \ref{alg:xMCMC}, which is a key indicator of the MCMC efficiency. 

\subsubsection{Pseudo-marginal MCMC}
\label{sec:pseudoanalysis}

We first analyze the pseudo-marginal method (Algorithm \ref{alg:pMCMC}). Recall that Algorithm \ref{alg:pMCMC} can be interpreted as an MCMC sampling method targeting the marginal posterior $\phi^y(z_r)$ using a low-dimensional proposal $p(z'_r | z_r)$, in which the acceptance/rejection is made based on a Monte Carlo estimate of the marginal posterior. Intuitively, the random estimate of the marginal posterior, of which the variance is controlled by the sample size $m$ and the LIS basis $U_r$, may lead to a loss of efficiency. Thus, we aim to compare the acceptance rate of Algorithm \ref{alg:pMCMC} with that of an idealized MCMC method that directly samples the (exact) marginal posterior. We present such an idealized algorithm below in terms of the marginal posterior in the reference coordinate. It uses the same low-dimensional proposal distribution as Algorithm \ref{alg:pMCMC}.

\begin{algorithm}[H]
   \SetAlgoLined
   \SetKwFunction{algo}{idealized-marginal-MCMC}
   \SetKwInOut{Required}{Required}
   \Required{marginal likelihood $\gbar(z_r;y)$ and LIS proposal $p(\cdot|\cdot)$ on $\R^r$.}
   \KwIn{current state $z_r$.}
   \KwOut{new state $z'_r$.}
   
   \SetKwProg{myalg}{Algorithm}{}{}
   \SetKwProg{myproc}{Procedure}{}{}
   
   \BlankLine
   \myalg{\algo{$z_r, \gbar, p$}}{
     Generate a LIS proposal candidate $z'_r\sim p(\,\cdot\,|z_r)$\;
     Compute the acceptance probability 
    \[
        \alpha_\ast(z_r,z_r')=1 \wedge \frac{\gbar(z'_r;y)\phi^0(z_r')p(z_r| z'_r)}{\gbar(z_r;y)\phi^0(z_r)p(z'_r | z_r)}\;
    \]
    \If(reject the proposal candidate){$\mathrm{uniform}[0,1] > \alpha_\ast$}{ Set $z'_r = z_r$\;}
     \KwRet $z_r^\prime$\;
   }
   \caption{One MCMC step to sample from $\phi^y(z_r)$}
   \label{alg:subspaceMCMC}
\end{algorithm}

It is shown in \cite{andrieu2015convergence} that the efficiency of the pseudo-marginal method asymptotically approaches that of the idealized Algorithm \ref{alg:pMCMC} as the variance of the Monte Carlo estimate of the marginal posterior decreases. Thus, we can apply the analysis of the LIS to investigate the acceptance rate of Algorithm \ref{alg:pMCMC} as follows. 

\begin{proposition}
The expected acceptance probabilities of Algorithm \ref{alg:pMCMC} and Algorithm \ref{alg:subspaceMCMC} 
satisfy the following inequality:
\[
0\leq \E[\alpha_\ast - \alpha_1] \leq 2\sqrt{\res(U_r , H)}. 
\]
In addition, if the conditional likelihood is bounded as $\frac{g(z;y)}{\bar{g}(z_r;y)}\leq C$, then 
\[
0\leq \E[\alpha_\ast - \alpha_1] \leq \frac{\sqrt{C}}{\sqrt{m}}\sqrt{\res(U_r , H)}.
\]
\end{proposition}
This result indicates that if we have selected a subspace so that $\res(U_r , H)$ is small, then the excepted acceptance probability of the pseudo-marginal is close to that of the idealized algorithm. 

\begin{proof}
For simplicity we write $g(z_r, z^i_\bot;y)=g( U_rz_r + U_\bot z^i_\bot;y)$.
Applying \cite[Corollary 4]{andrieu2009pseudo}, we have  
\begin{align*}
0\leq \E[\alpha_\ast-\alpha_1] &\leq\E \left|\frac{1}{m} \sum_{i=1}^m \frac{g(z_r, z^i_\bot;y)}{\gbar(z_r;y)}-1\right|.
\end{align*}
Denoting $q(z_\bot|z_r)=\frac{g(z_r,z_\bot;y)}{\bar{g}(z_r;y)}$, we have the identity
\begin{equation}
\E_{z_\bot\sim \phi^0}\big[ q(z_\bot|z_r)\big]=\int_{\reals^{d-r}}q(z_\bot|z_r) \phi^0(z_\bot) \d z_\bot=\frac{\int_{\reals^{d-r}}g(z_r,z_\bot;y)\phi^0(z_\bot) \d z_\bot}{\bar{g}(z_r;y)}=1. \label{eq:tmp431}
\end{equation}
For the first claim, we note that
\begin{align*}
\E \Bigg|\frac{1}{m} \sum_{i=1}^m \frac{g(z_r, z^i_\bot;y)}{\gbar(z_r;y)} -1 \Bigg| 
& \leq \frac{1}{m} \sum_{i=1}^m\E  \left|\frac{g(z_r, z^i_\bot;y)}{\gbar(z_r;y)}-1\right| \\
& = \E_{z_r\sim \phi^y}\Big[ \E_{z_\bot\sim \phi^0} \big[ |q(z_\bot|z_r)-1| \big] \Big]\\
& \leq \E_{z_r\sim \phi^y}\Big[ \sqrt{ \E_{z_\bot\sim \phi^0} \big[ |\sqrt{q(z_\bot|z_r)}+1|^2 \big] } \, \sqrt{ \E_{z_\bot\sim \phi^0} \big[ |\sqrt{q(z_\bot|z_r)}-1|^2 \big] } \Big].
\end{align*}
Plugging the identity 
\[
\E_{z_\bot\sim \phi^0}\big[ |\sqrt{q(z_\bot|z_r)}+1|^2 \big]\leq 
\E_{z_\bot\sim \phi^0} \big[ 2q(z_\bot|z_r)+2 \big] =4
\]
into the above inequality, we have
\begin{align}
    \E \Bigg|\frac{1}{m} \sum_{i=1}^m \frac{g(z_r, z^i_\bot;y)}{\gbar(z_r;y)} -1 \Bigg| & \leq 2 \E_{z_r\sim \phi^y}\Big[\sqrt{ \E_{z_\bot\sim \phi^0} \big[ |\sqrt{q(z_\bot|z_r)}-1|^2 \big] } \Big] \nonumber \\
    & \leq 2 \sqrt{ \E_{z_r\sim \phi^y}\Big[ \E_{z_\bot\sim \phi^0} \big[ |\sqrt{q(z_\bot|z_r)}-1|^2 \big] \Big]}. \label{eq:tmp432}
\end{align}
Then, we note that $0\leq\E_{z_\bot\sim \phi^0}[\sqrt{q(z_\bot|z_r)}]\leq 1$ by applying \eqref{eq:tmp431} and Jensen's inequality, which leads to 
\begin{align*}
\E_{z_\bot\sim \phi^0}\big[ |\sqrt{q(z_\bot|z_r)}-1|^2\big]
&=2\Big(1 -\E_{z_\bot\sim \phi^0}\big[\sqrt{q(z_\bot|z_r)} \big]  \Big)  \\
&\leq 2\Big(\E_{z_\bot\sim \phi^0}\big[q(z_\bot|z_r) \big] - \big(\E_{z_\bot\sim \phi^0}\big[\sqrt{q(z_\bot|z_r)} \big] \big)^2  \Big)    \\
&=2 \text{var}_{z_\bot\sim \phi^0} [\sqrt{q(z_\bot|z_r)}]
\end{align*}
Thus, the expectation in the right-hand-side of \eqref{eq:tmp432} satisfies 
\begin{align*}
\E_{z_r\sim \phi^y}\Big[ \E_{z_\bot\sim \phi^0} \big[ |\sqrt{q(z_\bot|z_r)}-1|^2 \big] 
&\leq 2\E_{z_r\sim \phi^y}\Big[ \text{var}_{z_\bot\sim \phi^0} [\sqrt{q(z_\bot|z_r)}] \Big]\\
&\leq 2\E_{z_r\sim \phi^y}\Big[ \E_{z_\bot\sim \phi^0}\big[ \|\nabla_{z_\bot}\sqrt{q(z_\bot|z_r)}\|^2 \big] \Big]\\
&=\E_{z_r\sim \phi^y}\Big[ \E_{z_\bot\sim \phi^0} \big[\|\nabla_{z_\bot}\log g(z;y)\|^2  q(z_\bot|z_r)\big] \Big]\\
&=\frac{1}{\calZ}\int_{\reals^d} \|U_\bot\nabla\log g(z;y)\|^2  \frac{g(z;y)}{\gbar(z_r;y)} \gbar(z_r;y) \phi^0(z) \d z\\
&=\int_{\reals^d} \|U_\bot\nabla\log g(z;y)\|^2  \phi^y(z)\d z=\res(U_r , H),
\end{align*}
where the second inequality above follows from the Poincar\'{e} inequality for $\phi^0$. Plugging the above inequality into \eqref{eq:tmp432}, the first claim follows.

For the second claim, the independence of $z^i_\bot$ yields
\[
\E  \left|\frac{1}{m} \sum_{i=1}^m \frac{g(z_r, z^i_\bot;y)}{\gbar(z_r;y)}-1 \right|\leq \frac1{\sqrt{m}}\sqrt{\E_{z_r\sim \phi^y} \Big[\text{var}_{z_\bot\sim \phi^0}[q(z_\bot|z_r)]\Big]}
\]
By Poincar\'{e} inequality of $\phi^0$, we have
\begin{align*}
\E_{z_r\sim \phi^y} \Big[ \text{var}_{z_\bot\sim \phi^0}[q(z_\bot|z_r)] \Big]&\leq \E_{z_r\sim \phi^y} \Big[ \E_{z_\bot\sim \phi^0}\big[\|\nabla_{z_\bot} q(z_\bot |z_r)\|^2\big] \Big]\\
&=\E_{z_r\sim \phi^y} \Big[ \E_{z_\bot\sim \phi^0} \big[ \|\nabla_{z_\bot} \log g(z; y)\|^2 q(z_\bot|z_r)^2 \big]\Big]\\
&=\frac{1}{\calZ} \int_{\reals^d}\|\nabla_{z_\bot} \log g(z; y)\|^2 q(z_\bot|z_r)g(z;y) \phi^0(z)\d z\\
&=\int_{\reals^d}\|U_\bot \nabla \log g(z;y)\|^2 \frac{g(z;y)}{\gbar(z_r;y)}\phi^y(z)\d z\\
&\leq C \int_{\reals^d}\|U_\bot \nabla \log g(z;y)\|^2 \phi^y(z)\d z=C \res(U_r , H).
\end{align*}
This concludes the proof.
\end{proof}

\subsubsection{Efficiency of delayed acceptance}
\label{sec:daanalysis}

The efficiency of the delayed acceptance method critically relies on the accuracy of the approximate density. We introduce the following proposition to analyze Algorithm \ref{alg:xMCMC}.
\begin{proposition}
\label{prop:delayprop}
For two probability densities $\mu$ and $\nu$. Suppose a Markov chain transition kernel $Q(\cdot|\cdot)$ satisfies the detailed balance condition with respect to $\mu$:
\[
    \mu(x)Q(x'|x) = \mu(x') Q(x|x') .
\]
Then, a Metropolis-Hastings algorithm defined by the proposal $Q(\cdot|\cdot)$ and the acceptance probability 
\[
\alpha(x,x')=1\wedge \frac{\nu(x')\mu(x)}{\mu(x')\nu(x)} 
\]
satisfies the detailed balance condition with respect to $\nu$. Moreover, the rejection rate satisfies
\[
\E \big[ 1-\alpha(X,X') \big]\leq 4\sqrt{2}\Hell(\nu,\mu),
\]
where $X$ is a random sample from $\nu$ and $X'$ is generated randomly from $Q(\cdot|X)$. 
\end{proposition}
\begin{proof}
For the detailed balance condition, we simply check that 
\begin{align*}
\nu(x)Q(x'|x)\alpha(x,x')&=\nu(x)\frac{\mu(x')}{\mu(x)}Q(x|x')\left[1\wedge \frac{\nu(x')\mu(x)}{\mu(x')\nu(x)}\right]\\
&=\nu(x')Q(x|x')\left[\frac{\nu(x)\mu(x')}{\mu(x)\nu(x')}\wedge 1\right]\\
&=\nu(x')Q(x|x')\alpha(x',x). 
\end{align*}
Meanwhile, let $b(x)=\frac{\nu(x)}{\mu(x)}$, the rejection rate can be expressed as
\begin{align*}
\E \big[ 1-\alpha(X,X') \big] &=\int_{\reals^d}\int_{\reals^d}\nu(x)Q(x'|x)(1-\alpha(x,x'))\d x \d x'\\
&=\int_{\reals^d}\int_{\reals^d}\mu(x)Q(x'|x)\left(1-\left(\frac{\nu(x)}{\mu(x)}\wedge \frac{\nu(x')}{\mu(x')}\right)\right)\d x \d x'\\
&=\int_{\reals^d} \int_{\reals^d}\mu(x)Q(x'|x)\left((1-b(x))\vee(1-b(x'))\right)\d x \d x' .
\end{align*}
Then note that for any $b\geq 0, 1-b\leq 2-2\sqrt{b}\leq |2-2\sqrt{b}|$, so 
\[
(1-b(x))\vee(1-b(x'))\leq |2-2\sqrt{b(x)}|\vee |2-2\sqrt{b(x')}|\leq |2-2\sqrt{b(x)}|+ |2-2\sqrt{b(x')}|.
\]
Applying the above identity and the detailed balance condition $\mu(x)Q(x'|x) = \mu(x') Q(x|x')$, we have
\begin{align*}
\E \big[ 1-\alpha(X,X')\big]&\leq \int_{\reals^d}\int_{\reals^d}\mu(x)Q(x'|x) |2-2\sqrt{b(x)}| \d x \d x'\\
&\quad +\int_{\reals^d}\int_{\reals^d}\mu(x)Q(x'|x) |2-2\sqrt{b(x')}| \d x \d x'\\
&\text{(Using $\mu(x)Q(x'|x)=\mu(x')Q(x|x')$ and symmetry, these two are the same)}\\
&= 2\int\int_{\reals^d}\mu(x)Q(x'|x) |2-2\sqrt{b(x)}| \d x \d x'\\
&= 4\int_{\reals^d}\mu(x) |1-\sqrt{b(x)}| \d x\\
&\leq 4\sqrt{\int_{\reals^d}\mu(x) |1-\sqrt{b(x)}|^2 \d x}=4 \sqrt{2} \Hell(\mu,\nu)
\end{align*}
where the second step follows from $\int_{\reals^d}Q(x'|x)\d x'=1$ and the last step follows from Jensen's inequality. This concludes the proof.
\end{proof}

Proposition \ref{prop:delayprop} provides a rigorous justification for Algorithm \ref{alg:xMCMC}. In particular, we have the following result. Since $\Hell(\pi^y,\pihat^y )$ can be bounded using Proposition \ref{prop:errT}, if $\That$ is a close approximation of $T$, then delay acceptance tends to be efficient. 

\begin{proposition}
Algorithm \ref{alg:xMCMC} satisfies the detailed balance condition with respect to the original posterior $\pi^y$. Moreover, the acceptance rate can be bounded from below by  
\[
\E \big[ \alpha_2(X,X') \big]\geq 1-4\sqrt{2}\Hell(\pi^y,\pihat^y ) .
\] 
\end{proposition}
\begin{proof}
We first note that Lines 2--4 of Algorithm \ref{alg:xMCMC} defines an MCMC transition kernel $Q(\cdot|\cdot)$ that satisfies the detailed balance condition with respect to the approximate posterior $\pihat^y$. Recall \eqref{eq:da_approx}, the ratio between the approximate posterior and the original posterior is given by 
\[
\frac{\pihat^y(x)}{\pi^y(x)} \propto \frac{\pihat^0(x)}{\pi^0(x)},
\]
where $\pihat^0(x) = \mathrm{det}(\nabla \That(\That^{-1}(x)) )\phi^0(\That^{-1}(x))$. Let $x=\That(z)$, we have
\[
\frac{\pihat^y(x)}{\pi^y(x)} \propto \frac{\mathrm{det}(\nabla \That(z) )\phi^0(z)}{\pi^0(x)}.
\]
Then applying Proposition \ref{prop:delayprop}, we can conclude that Algorithm \ref{alg:xMCMC} satisfies the detailed balance condition with respect to the original posterior $\pi^y$. The second claim directly follows from Proposition \ref{prop:delayprop}.

\end{proof}

\section{Numerical experiments}
\label{sec:numerics}
In this section, we will demonstrate our methods through two numerical examples. In the first example, we demonstrate the sampling performance and dimension scalability using a one-dimensional elliptic inverse problem. In the second example, apply our methods to a more complicated linear elasticity problem on an irregular domain.

\subsection{Elliptic inverse problem}
\label{sec:elliptic}
\subsubsection{Problem setup}
The first example is a one-dimensional PDE-constrained inverse problem with Gaussian measurement noise. We adopt here a similar setup as in \cite{bardsley2021optimization}: we aim to estimate the diffusion field $ s\mapsto\kappa(s) >0$  from measurements of the potential function $s\mapsto u(s)$ which solves the Poisson equation
\begin{equation}
\label{HeatIP}
-\frac{\partial}{\partial s}\Big( \kappa(s) \frac{\partial  u}{\partial s}(s)\Big)=f(s),\quad s\in \Omega := (0,1), 
\end{equation}
with boundary conditions $u(0)=u(1)=0$. 
We consider two different right-hand side functions $f_1 ,f_2$ that are scaled Dirac delta functions
\begin{equation*}
f_1(s)=1000\cdot\delta(s-1/3) \quad{\rm and}\quad f_2(s)=1000\cdot\delta(s-2/3).
\end{equation*}
Denoting by $u_1,u_2$ the solutions associated with $f=f_1 $ and $f=f_2$ respectively,  we generate the data set corresponding to the measurements of $u_1$ and $u_2$ at $31$ equally spaced discrete locations in $(0,1)$. The resulting data $y\in\R^{62}$ are given by
$$
 y = (u_1(s_1^\text{obs}),\hdots u_1(s_{31}^\text{obs}),u_2(s_1^\text{obs}),\hdots u_2(s_{31}^\text{obs})) + \varepsilon ,
$$
where $\varepsilon\sim\mathcal{N}(0,\sigma^2 I)$ with $\sigma^2$ corresponding to a signal-to-noise ratio of 10\%. 
The ``true'' diffusion coefficient $\kappa_{\rm true}$ used to generate the data is 
\begin{equation*}
\kappa_{\rm true}(s)= \left\{ \begin{array}{ll} 5, & s \in [0,0.2) \\ 1, & s \in [0.2,0.5) \\ 3, & s \in [0.5,0.75) \\ 5, & s \in [0.75, 1]. \end{array} \right. 
\end{equation*}

\subsubsection{Discretization} We use the finite element method with $d=2^\ell$ uniform elements to discretize the equation. $\kappa(s)$ is estimated by a piecewise-constant field with value $\kappa_i$ at the $i$-th element and $u(s)$ is estimated by a continuous piecewise-linear function with value $u_i$ at the $i$-th node ($d+1$ nodes total). Thus, \eqref{HeatIP} can be reformulated as finding $u\in\R^{d-1}$ such that
$
B(x) u = f,
$
where $f = (f(s_1) ,\hdots, f(s_{d-1}))$ and
$$
 B(x) = \begin{pmatrix} \kappa_{1}+\kappa_{2} & -\kappa_{2} & & 0 \\ -\kappa_{2} & \ddots & \ddots & \\ &\ddots&\ddots& -\kappa_{d-1}\\0 &&-\kappa_{d-1}& \kappa_{d-1}+\kappa_{d}\end{pmatrix},
$$
with $\kappa_i =\log(\exp(x_i) + 1)$. In the end, the discretized forward model is
$$
 G(x) = \begin{pmatrix} C B(x)^{-1} f_1 \\ C B(x)^{-1} f_2 \end{pmatrix} ,
$$
where $C\in\R^{31\times (d-1)}$ is an observation matrix which extract the evaluations of $u$ at locations $s_{i}^\text{obs}$.
In Figure \ref{fig:ellptic_setup}, the ``true'' diffusion coefficient (generated with $\ell=10$) is shown in the top left plot and the corresponding data vectors together with the noise-free data $C\,B(x)^{-1}f_1$ and $C\,B(x)^{-1}f_2$ are shown in the top right plot. 

\subsubsection{Prior} We now describe prior models we put on the diffusion coefficient. In order to ensure the positivity of the diffusion field, we parametrize $\kappa(s)$ as
\[
\kappa(s) = \log(\exp(z(s)) + 1) ,
\]
where $s\mapsto z(s)$ is a random field. Then we consider two prior distributions for the random field $z(s)$. Firstly, we consider a Besov-type prior for $x$, meaning that $z = \sum_{i\geq1} X_i \psi_i $ where $(\psi_1,\psi_2,\hdots)$ is the Haar wavelet basis in $L^2((0,1))$ and where $X_1,X_2,\hdots$ are independent random variables following a exponential power distribution with $p\in\{0.5;1;2\}$, see Example \ref{ex:Exponential}.

Secondly, we consider a first-order difference Cauchy prior \cite{chada2021cauchy}. Following the piecewise constant discretization of the random fields $\kappa(s)$, and hence $z(s)$, we assign prior density to the difference between two adjacent elements of the discretized vector $z = (z_1, z_2, \ldots z_d)^T$. This leads to the unnormalized prior density on the vector $z$
\[
\pi^0(z) \propto \pi^0\left(z_1\right) \prod_{j=2}^d \pi^0\left(\frac{z_j - z_{j-1}}{h}\right),
\]
where $h = d^{-1}$ is the size of each local element. Equivalently, we can define a linear transformation $x = D z$, where $D\in \R^{d\times d}$ is a first order difference matrix
$$
 D = \begin{pmatrix} 1 & 0 & & & \\ -1/h & 1/h & 0 & & & \\  & -1/h & 1/h &\ddots  &  \\ & &\ddots & \ddots & 0 \\ & & & -1/h& 1/h\end{pmatrix},
$$
and then elements of the random vector $X = D^{-1}Z$ are independent random variables following the Cauchy distribution (cf. Example \ref{ex:Cauchy}). 

\begin{figure}[h!]
\includegraphics[trim = 2em 0em 2em 0em , width = 0.45\textwidth, height = 0.27\textwidth]{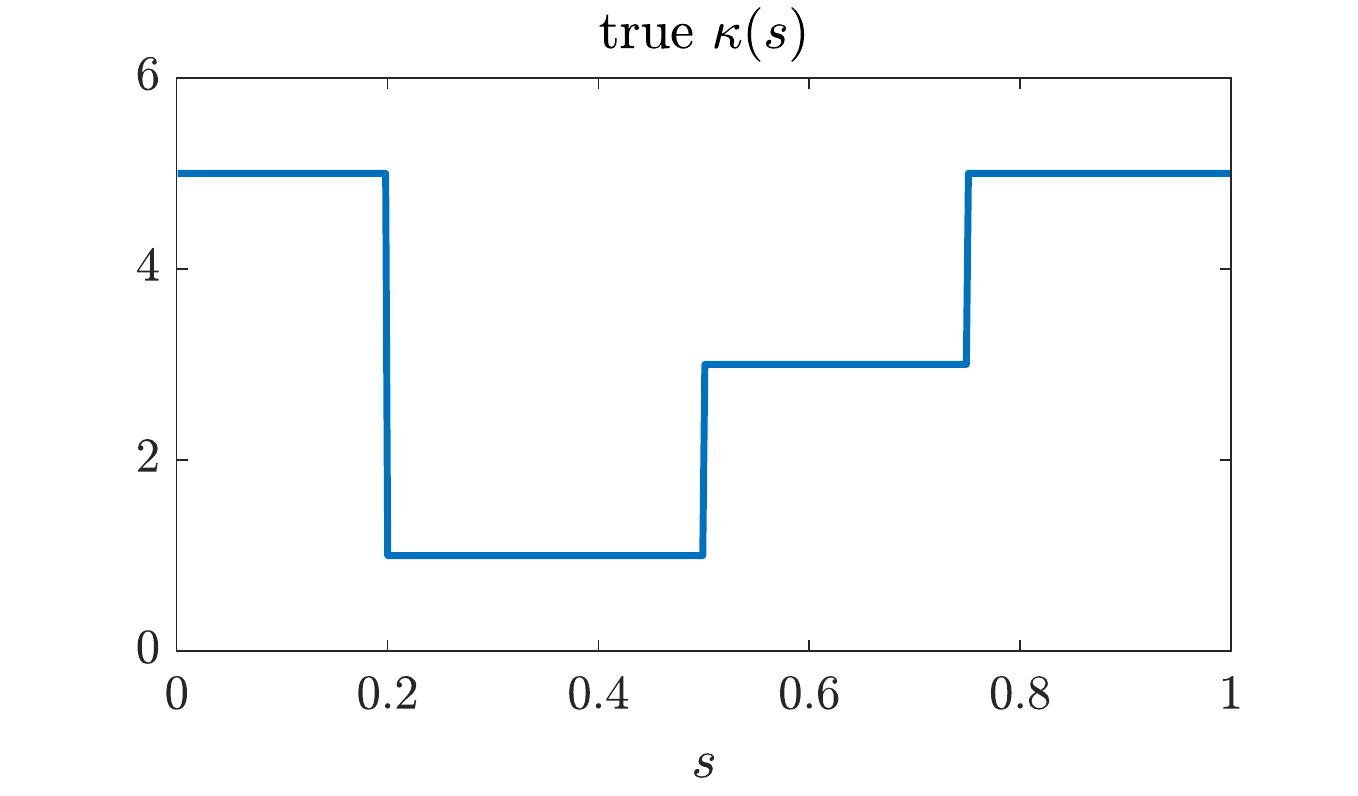}\hfill
\includegraphics[trim = 2em 0em 2em 0em , width = 0.45\textwidth, height = 0.27\textwidth]{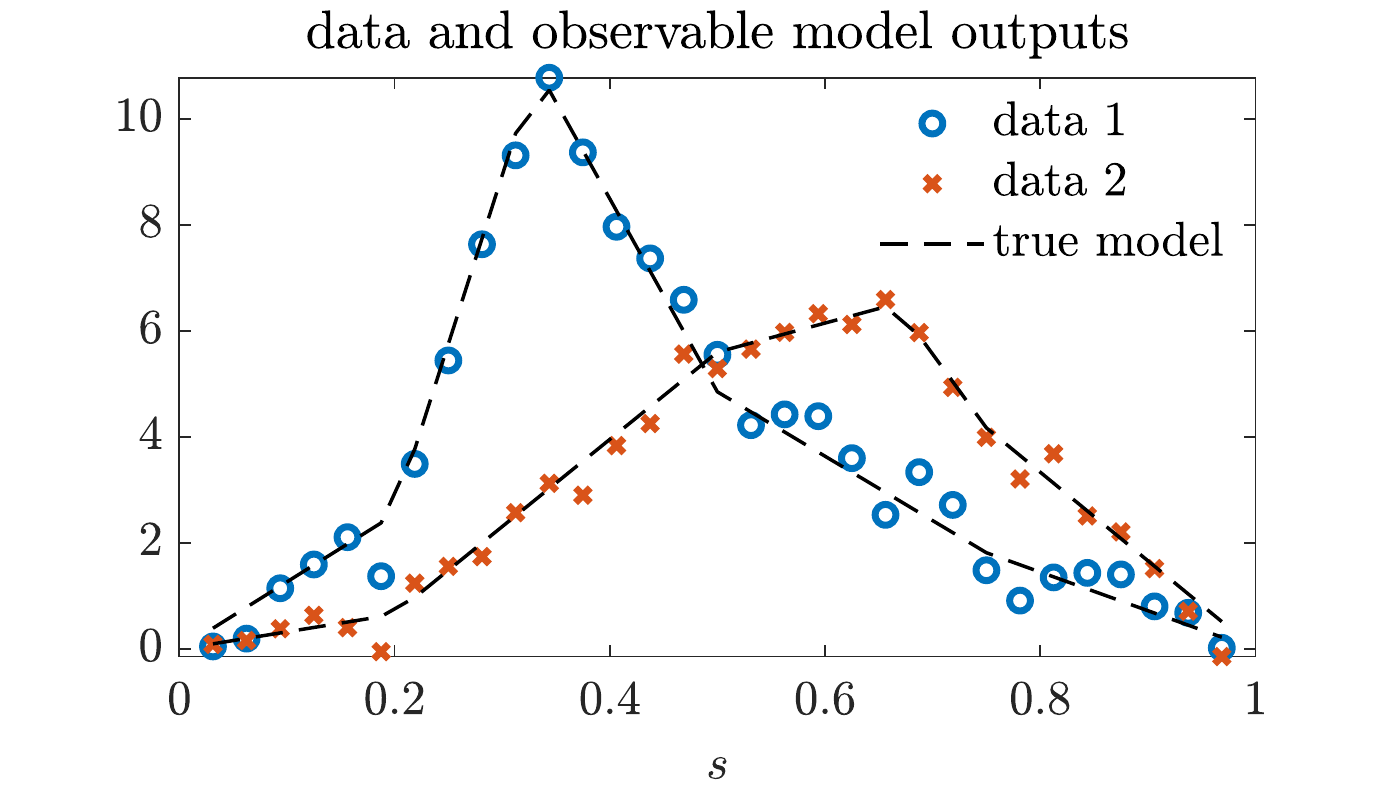}
\caption{Left: the ``true'' diffusion coefficient used for generating the synthetic data. Right: the measured data for the elliptic PDE problem. }\label{fig:ellptic_setup}
\end{figure}

\begin{figure}[h!]
\includegraphics[trim = 2em 0em 2em 0em , width = 0.45\textwidth, height = 0.27\textwidth]{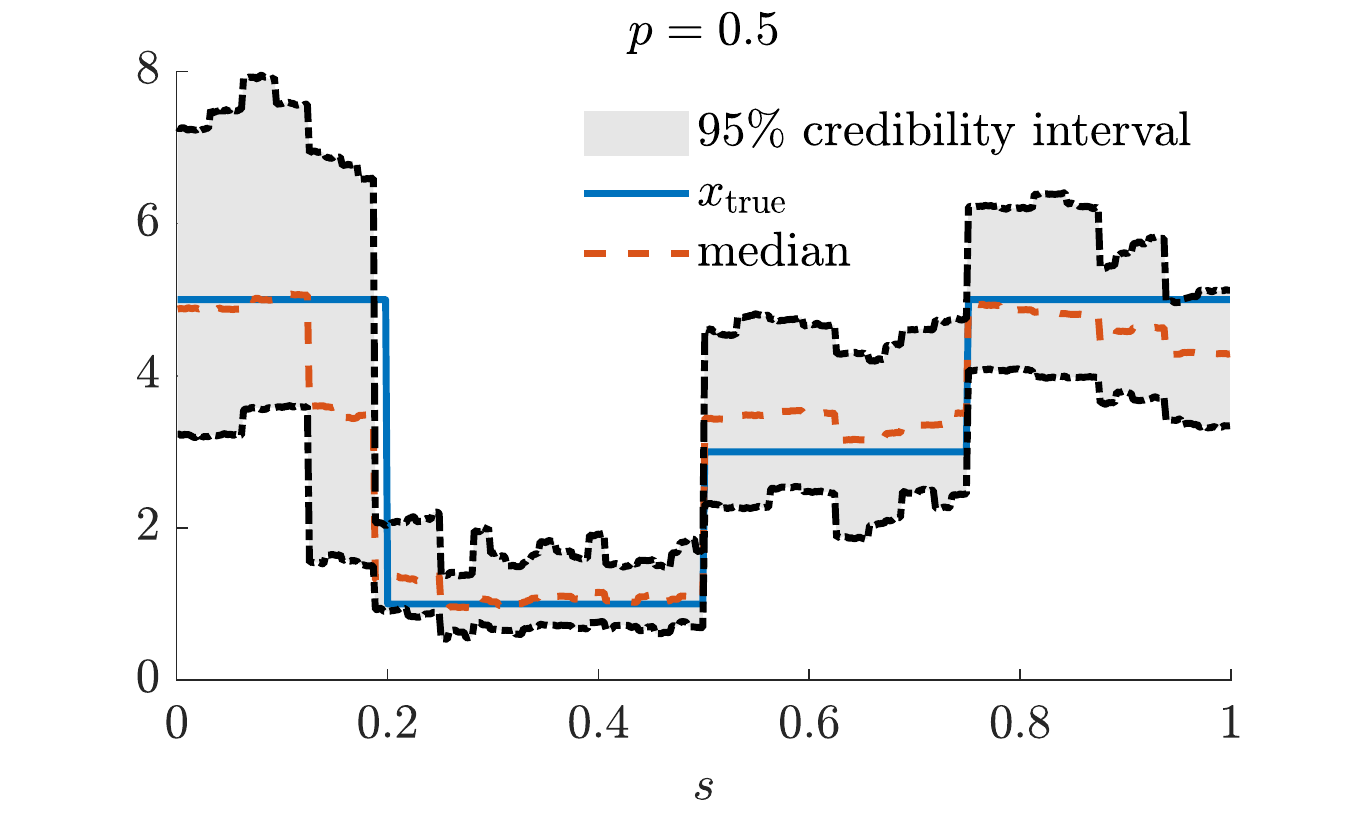}\hfill
\includegraphics[trim = 2em 0em 2em 0em , width = 0.45\textwidth, height = 0.27\textwidth]{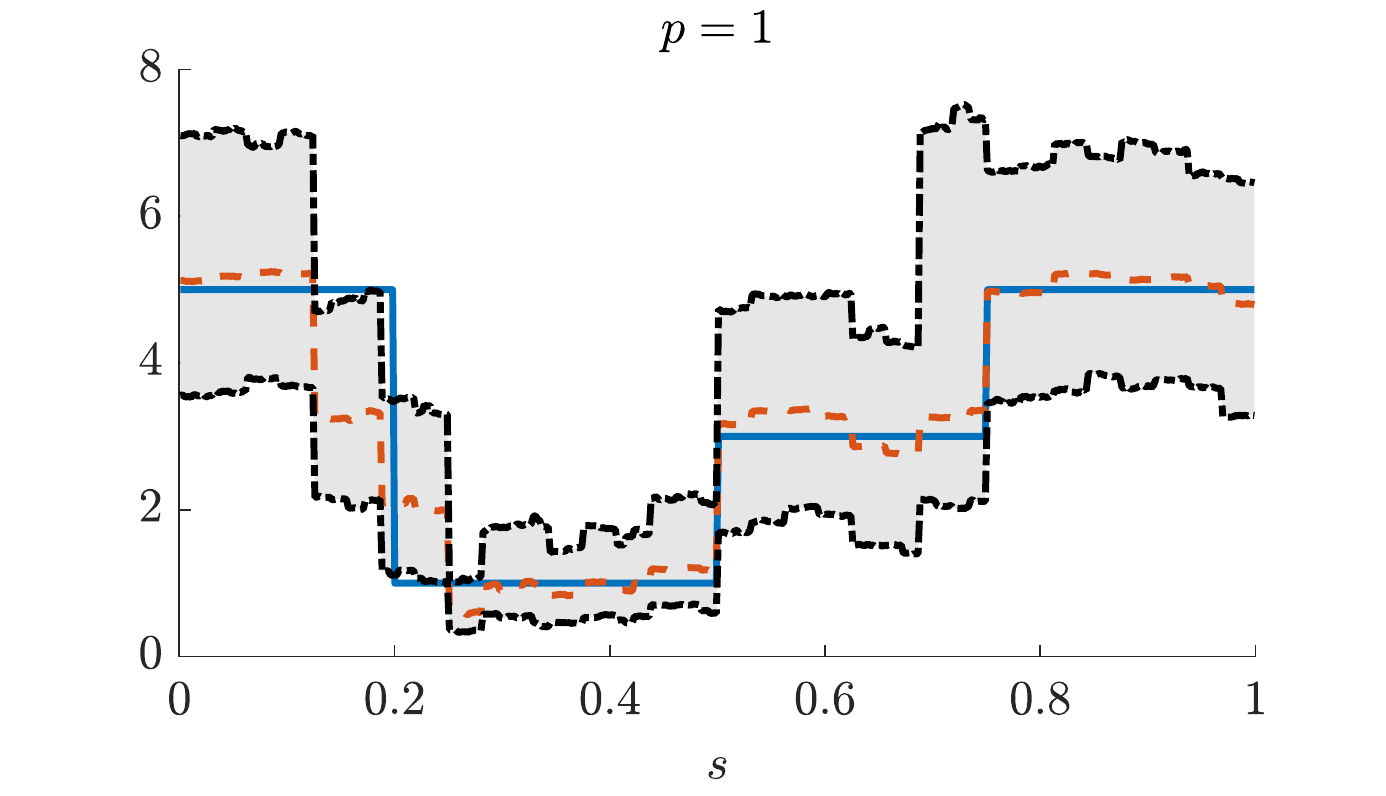}\\
\includegraphics[trim = 2em 0em 2em 0em , width = 0.45\textwidth, height = 0.27\textwidth]{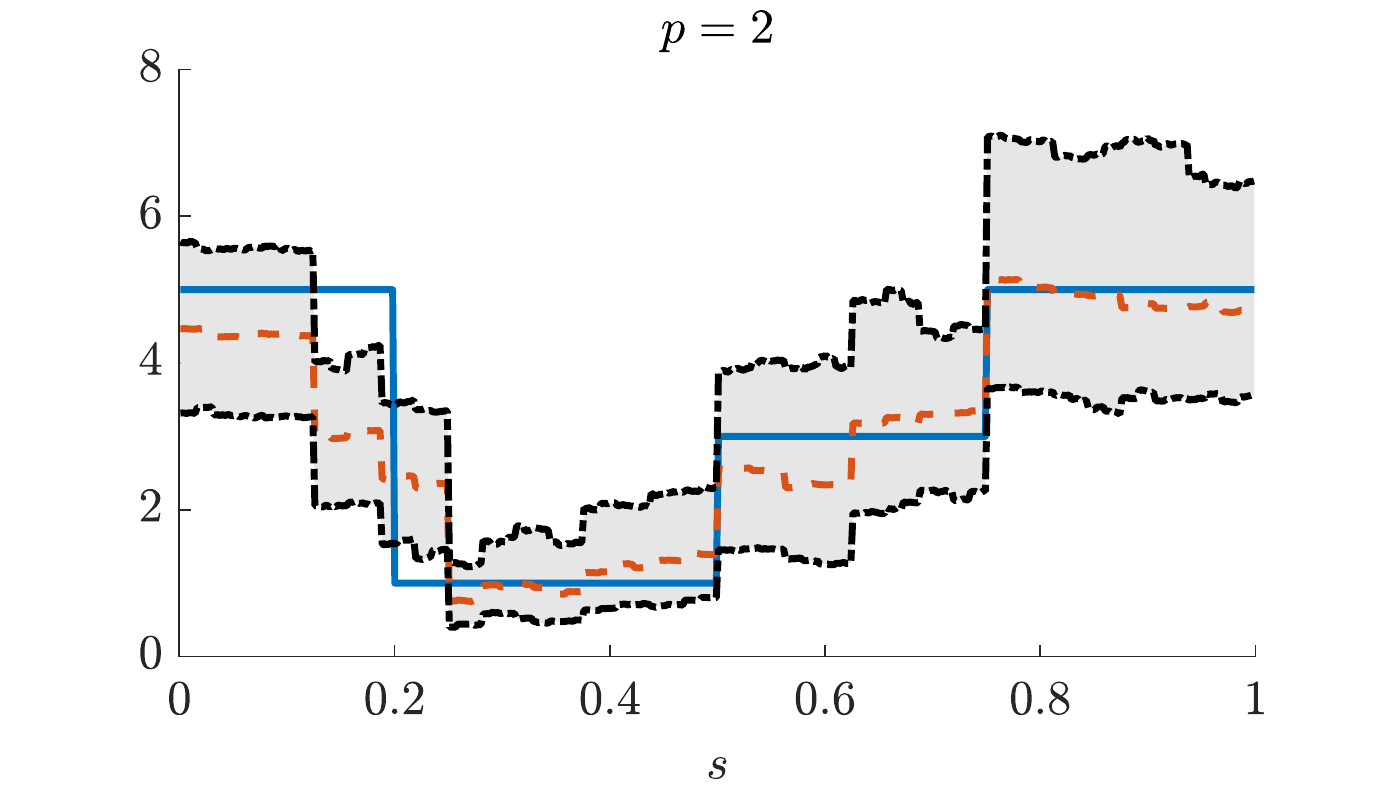}
\caption{Inversion results using the Besov-type prior. Starting from the left: medians and $95\%$ credibility intervals of the reconstructed diffusion coefficients using the exponential power distribution with $p = 0.5$, $p = 1$ (Laplace), and $p = 2$ (Gaussian), respectively.}\label{fig:ellptic_haar}
\end{figure}

\subsubsection{Numerical results of Besov-type priors}
In Figure \ref{fig:ellptic_setup}, we first show on the left  the true diffusion coefficient $\kappa$ to be recovered,  and the data obtained through two forcing profiles on the right. We have implemented MALA, pCN and Hamiltonian MCMC algorithms under the framework of Algorithm \ref{alg:pMCMC} with different subspace dimensions. In this example, we find that various MCMC algorithms only differ in efficiency, but not in the inversion results. So we will discuss the sampling efficiency next and only present the results obtained with MALA for different priors. The inversion results obtained by the Besov-type priors with power indices $p = \{0.5, 1, 2\}$ are shown in Figure \ref{fig:ellptic_haar}, and the inversion results obtained by the first-order difference priors are shown in Figure \ref{fig:ellptic_diff}.
For the Besov-type priors, we can see that with all three choices of $p$, the median estimators can recover largely the profile of true $\kappa$. However, we note that when using a Gaussian $(p=2)$ prior, the $95\%$ CI does not cover the true $\kappa(s)$ near $s=0.15$ and $\kappa=0.22$. Moreover, the CI cannot find  the ``change points" of  $\kappa(s)$ that accurately, missing the one near $s=0.2$. In comparison, using $p=1$ and $p=0.5$ yields CIs that cover the true $\kappa$. The CIs also capture the changes more accurately. This demonstrates the importance of using heavy-tailed priors.

\paragraph{Sampling performance}
We compare the performance of different subspace MCMC methods---MALA, pCN and Hamiltonian (with the number of integration steps automatically tuned by NUTS)---with their full-dimensional counterparts. 
For Besov-type priors used in this example (cf. Section 5.1.3), the posterior distribution equips with heavier prior tails, which corresponds to a lower value $p$ in the exponential power distribution, is more challenging to sample from. Thus, we use the most difficult case, $p = 0.5$, in our performance benchmarks.
For each of the subspace MCMC algorithms, we provide performance estimates using different subspace dimensions ($r = \{24, 32, 40\}$) and pseudo-marginal sample sizes ($m = \{2, 5\}$). With LIS dimensions $r = \{24, 32, 40\}$, the estimated trace residuals are $\res(U_r , \Hhat) = \{1.3, 0.61, 0.34\}$, respectively. For the subspace versions of pCN and MALA, we precondition the subspace proposals using the empirical marginal posterior covariance matrix adaptively estimated from past samples \cite{haario2001adaptive}, which is feasible for the rather low-dimensional LIS used here. For all of pCN, MALA and NUTS applied to the full posterior, we sample the posterior transformed to reference coordinate instead of directly sampling the original posterior. 

To measure the sampling efficiency, we use the average integrated autocorrelation times (IACTs) \cite[Chapter 11]{Owen} of parameters 
\[
\tau = \frac1d \sum_{i =1}^d \mathrm{IACT}(x_i),
\]
where $\mathrm{IACT}(x_i)$ is the IACT of the $i$th component of $x$. Roughly speaking, IACT can be seen as the number of steps needed for the algorithm to generate an uncorrelated new sample. So a smaller IACT indicates an algorithm having better sampling efficiency. Different MCMC algorithms used here require different computational costs per iteration. For pCN and MALA, each MCMC iteration only requires one (Monte Carlo) posterior density evaluation. For NUTS, each iteration requires a varying number of Hamiltonian steps so that the resulting end state of the Hamiltonian dynamics is sufficiently decorrelated from the starting point. For this reason, each iteration of NUTS is significantly more costly (about two orders of magnitude) than pCN and MALA. So we use the IACT in terms of the Hamiltonian steps taken as a performance indicator of NUTS. We note that the computational effort  needed by each Monte Carlo estimation of the marginal posterior in Algorithm \ref{alg:pMCMC} is about $m-1$ times more than that of the full posterior evaluation. However, we do not consider this factor in the reported IACTs, as the computation needed by the Monte Carlo estimation of the marginal posterior can be embarrassingly parallelized.

\begin{table}
    \centering
    \caption{Elliptic example with the Besov-type prior. Average IACTs of parameters computed by various implementations samplers. For pCN and MALA, the IACTs are measured in terms of MCMC iterations. For NUTS, the IACTs are given in terms of the number of Hamiltonian steps (which is close to the wall-clock time performance). All the data reported here are in the form of mean$\pm$standard derivation. Here the symbol ``$-$'' denotes Markov chains having unstable mixing behaviour in which the relative standard deviation of IACT exceeds $1$. Here $m$ is the pseudo-marginal sample size.}
    \label{table:iact_elliptic}
    \begin{tabular}{ll|cc|c}
    \hline
    && \multicolumn{2}{c|}{pseudo-marginal} & \multirow{2}{*}{full} \\ \cline{3-4}
    && $m=2$  & $m=5$ & \\ \hline
    \multirow{3}{*}{MALA} & $r=24$ & $13.2\!\pm\!1.1$ & $12.4\!\pm\!0.93$ & \multirow{3}{*}{$7930\!\pm\!270$} \\
    & $r=32$ & $17.2\!\pm\!3.0$ & $12.5\!\pm\!0.35$ & \\
    & $r=40$ & $17.6\!\pm\!1.3$ & $15.7\!\pm\!1.9$  & \\\hline
    \multirow{3}{*}{pCN}& $r=24$ & $14.8\!\pm\!0.54$ & $14.0\!\pm\!1.5$ & \multirow{3}{*}{$8820\!\pm\!220$} \\
    & $r=32$ & $17.1\!\pm\!0.96$ & $18.0\!\pm\!3.4$ & \\
    & $r=40$ & $18.4\!\pm\!0.55$ & $18.6\!\pm\!2.6$ & \\\hline
    \multirow{3}{*}{NUTS} & $r=24$ & $-$ & $14.5\!\pm\!1.6$  & \multirow{3}{*}{$142\!\pm\!5.5$} \\
    & $r=32$ & $13.3\!\pm\!1.2$  & $11.2\!\pm\!0.42$ & \\
    & $r=40$ & $12.5\!\pm\!0.73$ & $10.5\!\pm\!0.39$ & \\\hline
    \end{tabular}
\end{table}

We present the IACTs in Table \ref{table:iact_elliptic}. In this example, even a moderate LIS dimension can significantly improve the sampling efficiency. In particular, the IACT for vanilla MALA is 7930. Using a LIS dimension 24, the subspace counterpart with a pseudo-marginal sample size $2$ is able to reduce the IACT to $13$. A similar speedup is also observed for pCN. For both MALA and pCN, the pseudo-marginal sample size does not make a significant impact on the sampling performance. The reason can be that the residual $\res(U_r , \Hhat)$ is small for the LIS dimension used here. 
Some readers may consider that the poor performance of the vanilla MALA and pCN (directly applied to the full posterior) is a known fact in the literature. To compare with the state-of-the-art algorithms, we also implemented the Hessian-preconditioned MALA (H-MALA) of \cite{petra2014computational} for this example. H-MALA produces an estimated IACT of $68.9\pm7.0$, which is still considerably less efficient than subspace methods. 
The subspace implementation is also able to accelerate NUTS. However, we note that using a pseudo-marginal sample size $2$ with a LIS dimension $24$ does not lead to a reliable estimate of the IACT. The reason can be that the Hamiltonian MCMC is more sensitive to the variance of the Monte Carlo estimate of the marginal posterior. With either increasing LIS dimension or increasing pseudo-marginal sample size, we are able to stably obtain accelerated sampling performance. 

\paragraph{Dimension scalability}
We also refine the model discretization by setting $d = 2^\ell$ with $\ell \in\{9, 11, 13\}$ to demonstrate the dimension scalability of our subspace methods. In Figure \ref{fig:H_haar}, we plot the spectra of the estimated $H$ matrices with different $\ell$. The $H$ matrices are estimated using $10^4$ posterior samples to avoid random fluctuations. To focus on the dominating part of the spectra, we plot the first $2^9$ eigenvalues. In this figure, we observe that with increasing discretization dimensions, the spectra of the $H$ matrices are similar. This suggests that the LIS dimension is invariant with respect to model refinement in this example. Then, we simulate the subspace MALA methods for difference refinement factors $\ell \in\{9, 11, 13\}$. Here we keep the LIS dimension to 40 and use a pseudo-marginal sample size $m=5$. The estimated IACTs are $16.5\!\pm\!2.1$,  $15.8\!\pm\!2.4$ and $15.6\!\pm\!1.8$ for $\ell = 9, 11, 13$, respectively. Again, for each of the priors, our proposed method shows dimension invariant sampling performance.

\begin{figure}[h!]
\centering
\includegraphics[trim = 2em 0em 2em 0em , width = 0.45\textwidth, height = 0.27\textwidth]{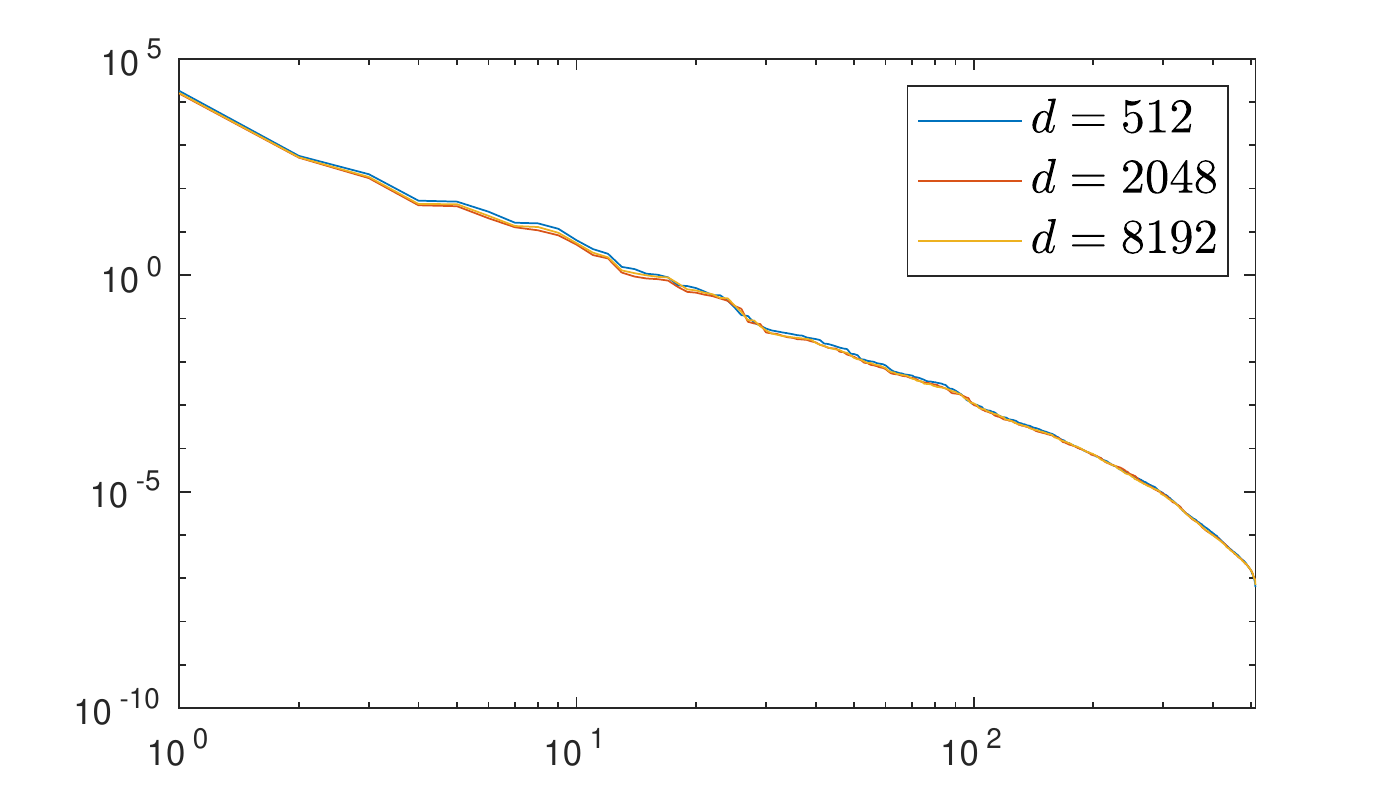}
\caption{Spectra of the estimated $H$ matrices using the Besov-type prior.}\label{fig:H_haar}
\end{figure}

\paragraph{Impact of prior normalization} As a final remark, we also provide sample histories of Markov chains generated by subspace MALA, vanilla MALA, and the vanilla MALA without prior normalization in Figure  \ref{fig:ellptic_MCMC}. Here we observe that the prior normalization is not only the enabling tool to build the highly efficient subspace MCMC, but also able to improve the mixing of vanilla MALA in this example. When NUTS is applied to sample the original posterior without prior normalization, the resulting Markov chain gets stuck in the initial state, and hence the result is not reported in this comparison. 

\begin{figure}[h!]
    \includegraphics[trim = 2em 0em 2em 0em , width = 0.45\textwidth, height = 0.27\textwidth]{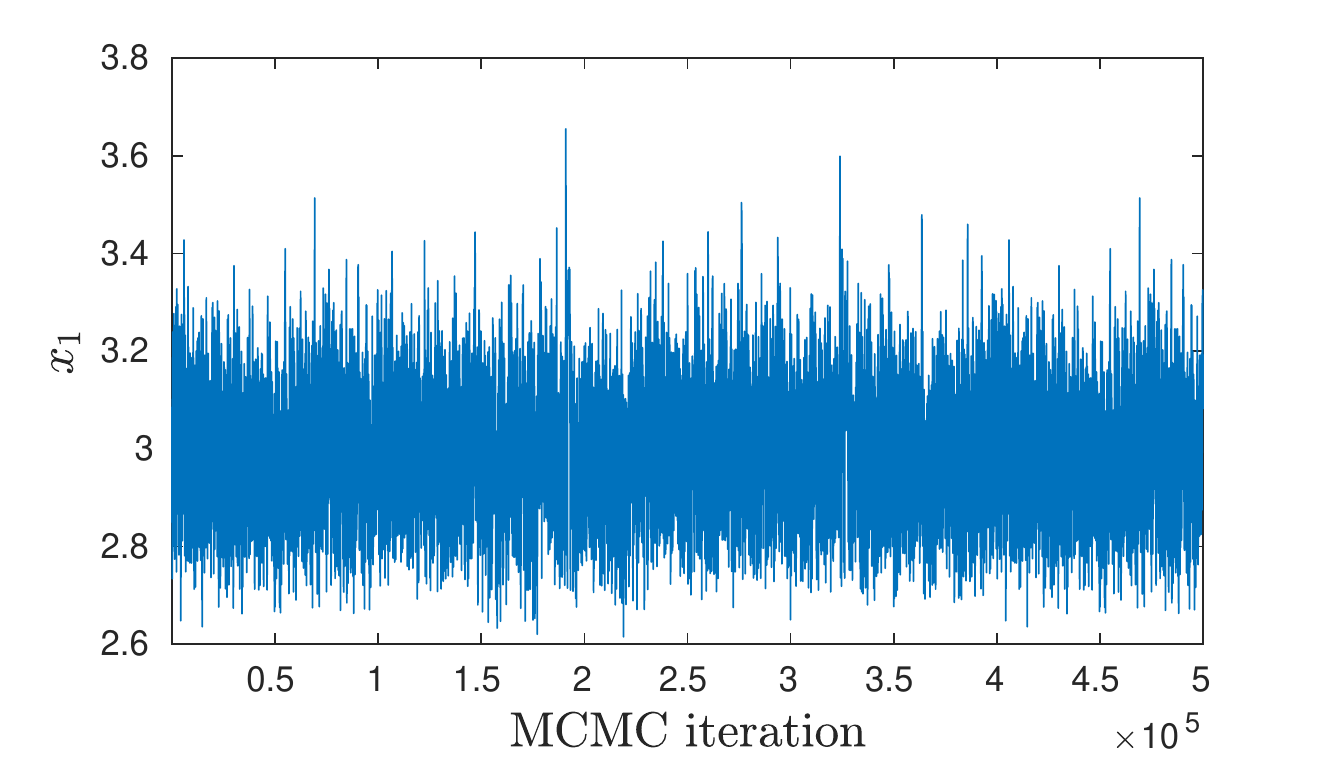}\hfill
    \includegraphics[trim = 2em 0em 2em 0em , width = 0.45\textwidth, height = 0.27\textwidth]{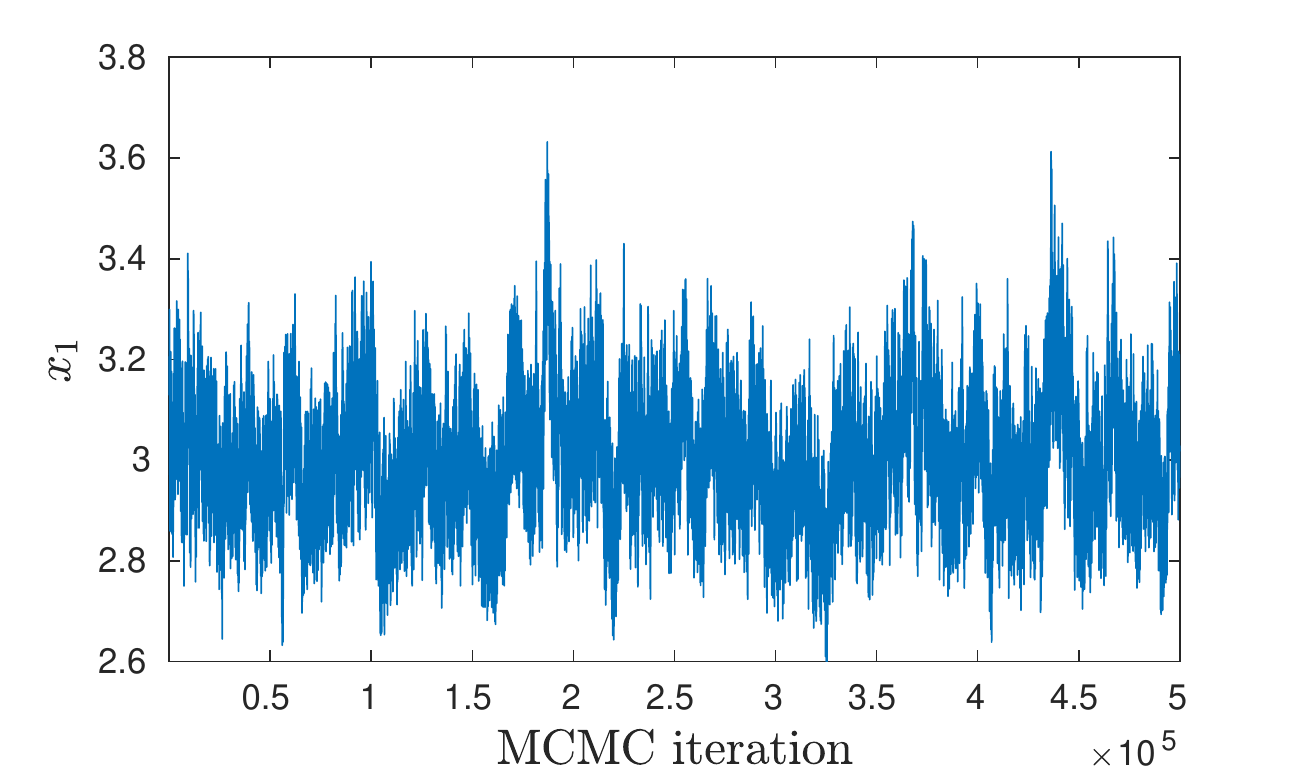}\\
    \includegraphics[trim = 2em 0em 2em 0em , width = 0.45\textwidth, height = 0.27\textwidth]{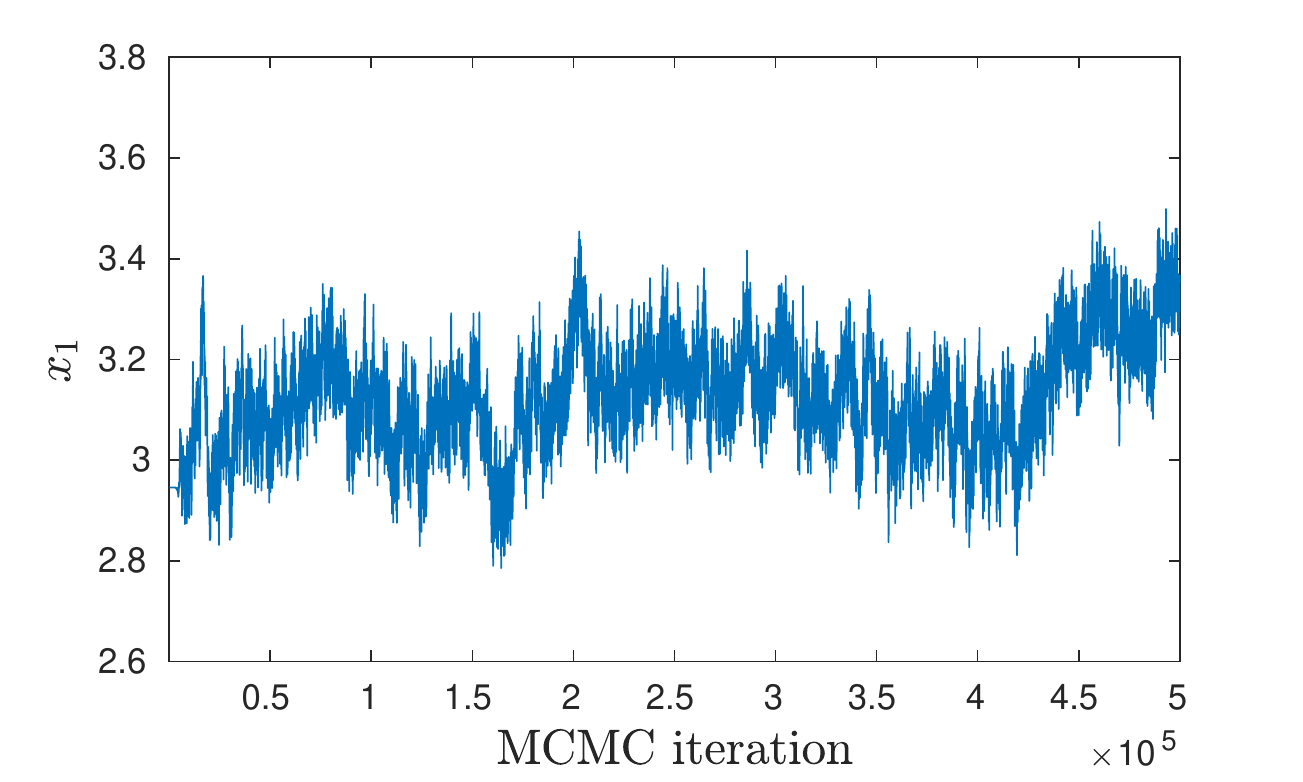}
    \caption{History of Markov chains of $x_1$ generated by subspace MALA with prior normalization (top left), vanilla MALA with prior normalization (top right) and vanilla MALA without prior normalization (bottom left). }\label{fig:ellptic_MCMC}
\end{figure}

\subsubsection{Numerical results of first-order difference Cauchy priors}

\begin{figure}[h!]
\centering
\includegraphics[trim = 2em 0em 2em 0em , width = 0.45\textwidth, height = 0.27\textwidth]{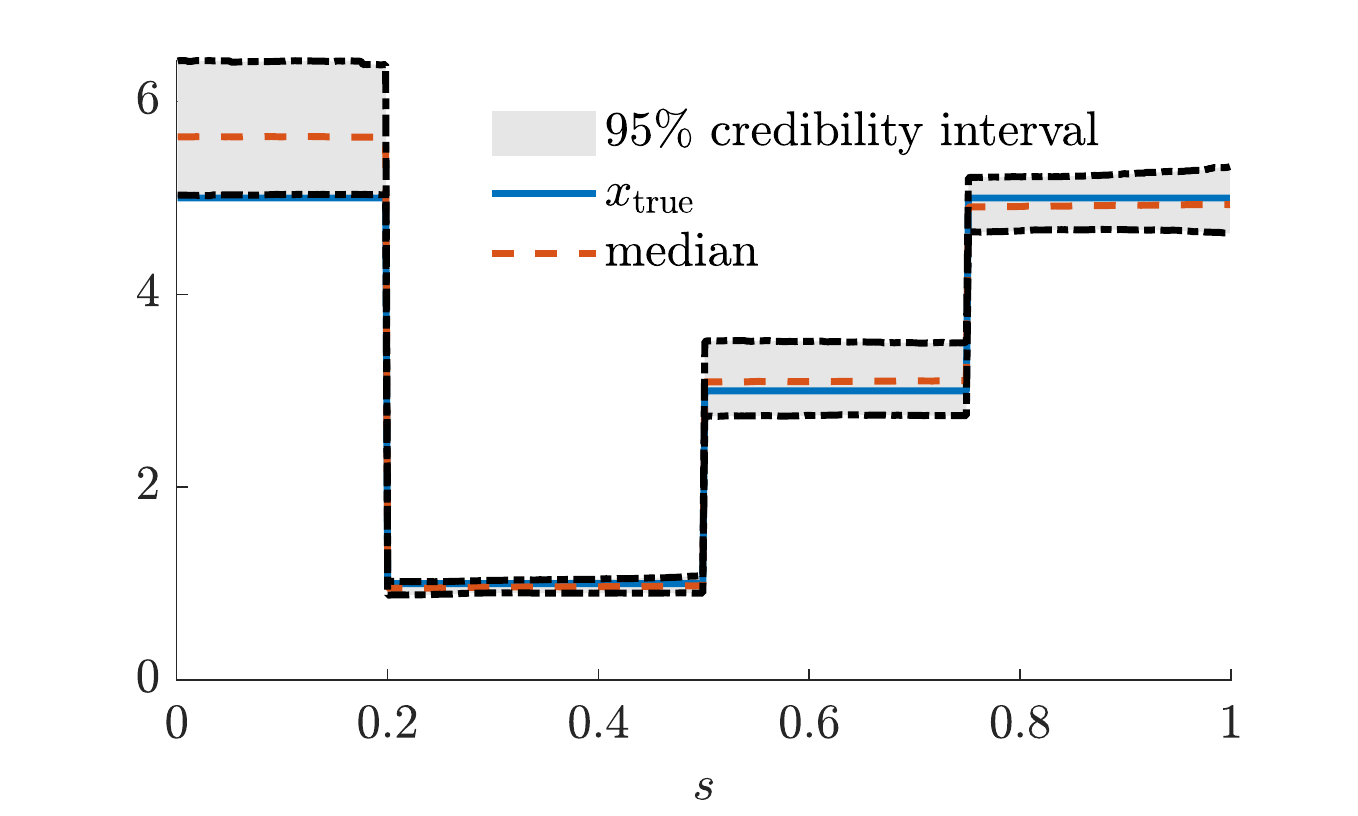}\hfill
\includegraphics[trim = 2em 0em 2em 0em , width = 0.45\textwidth, height = 0.27\textwidth]{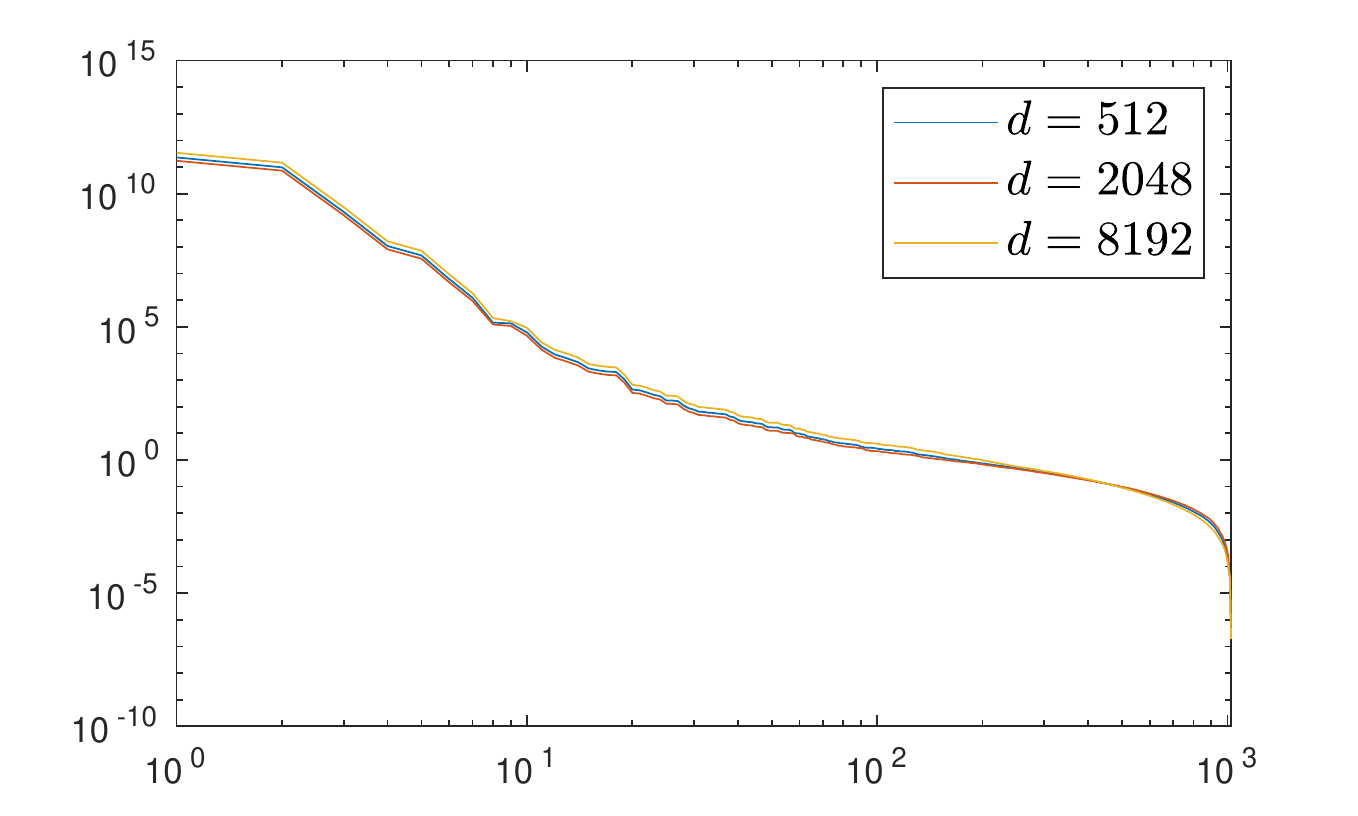}
\caption{Left: inversion results using the first-order difference Cauchy prior with $\ell = 9$ ($d = 512$): medians and $95\%$ credibility intervals of the reconstructed diffusion coefficients. Spectra of the estimated $H$ matrices with different dimensions. }\label{fig:ellptic_diff}
\end{figure}

Here we discuss the numerical results obtained using the Cauchy prior. Because the Cauchy prior generates local discontinuities with extreme values, the PDE solver may not have well-defined solutions. As a consequence, the condition number of the discretized model can be numerically infinite.
To handle this situation, we assign zero likelihood values to those numerically infeasible realizations of random fields.
This truncation may create boundaries with discontinuous posterior density in the parameter space, and thus make the posterior density challenging to sample from. 
As a result, in this example, MCMC simulations using all the abovementioned families of samplers cannot generate rapidly convergence Markov chains. 
To illustrate this, the full-space MALA and the subspace MALA produce Markov chains with estimated IACTs $4.9 \times 10^4 \pm 5.4 \times 10^2$ and $1.1 \times 10^3 \pm 0.7 \times 10^2$, respectively (here we used  $\ell = 9$). 

In Figure \ref{fig:ellptic_diff}, we observed that the Cauchy prior overall produce smaller uncertainty intervals. However, the Cauchy prior provides a better estimation of the location of the discontinuity compared to the Besov-type prior previously shown. 
We also note that the spectra of the estimated $H$ matrix are similar across different dimension settings. This suggests that the LIS dimension in this example should be invariant with respect to model discretization.

\subsection{Linear elasticity analysis of a wrench} 

In the second experiment, we consider a two-dimensional linear elasticity problem \cite{lam2020multifidelity,smetana2020randomized,baptista2022gradient} that models the displacement field $u:\mathcal{D} \rightarrow \mathbb{R}^{2}$ using the PDE 
\[
\nabla \cdot \big(K(s):\varepsilon(u(s))\big) = f(s), \quad s \in \mathcal{D} \subset \mathbb{R}^{2}.
\]
This equation is used to model the stress equilibrium in a physical body $\mathcal{D}$ subject to external forces. The physical body $\mathcal{D}$ is a wrench shown in Figure~\ref{fig:wrench}. Here, $\varepsilon(u) = \frac{1}{2}(\nabla u + \nabla u^{\top})$ is the strain tensor, and $s\mapsto K(s)$ is the Hooke tensor such that
\[
K(s):\varepsilon(u(s)) = \frac{E(s)}{1+\nu} \varepsilon(u(s)) + \frac{\nu E(s)}{1-\nu^2} \mathrm{trace}\big(\varepsilon(u(s))\big)\begin{pmatrix} 1&0 \\0&1\end{pmatrix}
,\]
where $\nu=0.3$ is Poisson's ratio and $s \mapsto E(s)$ is spatially varying Young's modulus such that $E(s) > 0$ for $\forall s$.
In this example, we aim to estimate Young's modulus, which is modeled by a real-valued random field $s\mapsto x(s)$ and the exponential function
\[
E(s) = \exp(x(s)).
\]
Here the random field $x(s)$ is represented as a linear combination of the eigenfunctions of the kernel $C(s,s') = \exp(-\|s-s'\|_2^2)$ on $\mathcal{D}\times \mathcal{D}$ as follow:
\[
x^h(s) = \sum_{i = 1}^d \psi_i^h(s) x_i,
\]
where $d=925$ is the number of elements in the mesh, $\{\psi_1^h, \ldots, \psi_d^h\}$ are the piecewise constant approximations to the eigenfunctions of $C(s,s')$, and the vector $x=(x_1,\hdots,x_d)$ is the unknown random coefficient to be estimated. Here, we prescribe a Laplace prior to the coefficients $x_1,\hdots,x_d$.
We compute the numerical solution $u^{h}=u^{h}(x^h)$ by Galerkin projection onto the space of continuous piecewise affine functions over a triangular mesh, see \cite{zienkiewicz2000finite}.
The domain $\mathcal{D}$, the mesh, the boundary conditions, and a sample von Mises stress of the solution are shown in Figure~\ref{fig:wrench}. We observe the vertical displacements $u^h_2$ at $26$ points of interest located along the green line where the force $f$ is applied, see Figure~\ref{fig:wrench}.
The perturbed observations are $y=u^h_2+e$ where $e$ is a zero-mean $H^{1}$-normal noise with the signal-to-noise-ratio $10$. Various summary statistics of the estimated parameters are shown in Figure \ref{fig:wrench_result}, where the posterior samples are obtained using subspace MALA. In this example, the posterior distribution is able to significantly reduce the prior uncertainty. 

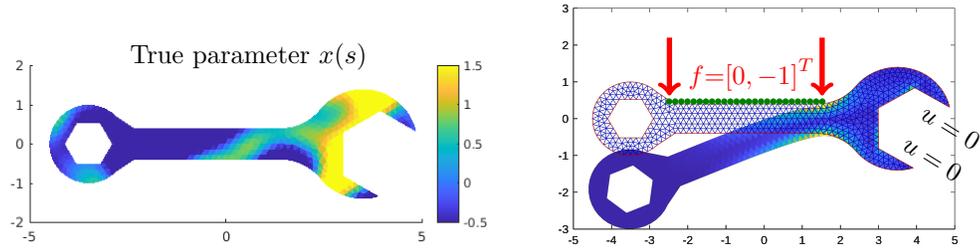
\begin{figure}
\begin{center}
\begin{tikzpicture}
    \node at (0,0) {\includegraphics[width=0.5\linewidth, trim = {0, 9em, 0, 6em}, clip]{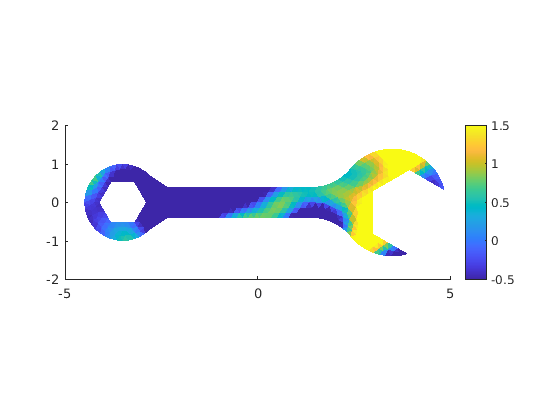}};
    \node at (0,+1) {True parameter $x(s)$};
\end{tikzpicture}
\begin{tikzpicture}
    \node at (0,0.25) {\includegraphics[width=0.35\linewidth,height=0.21\linewidth]{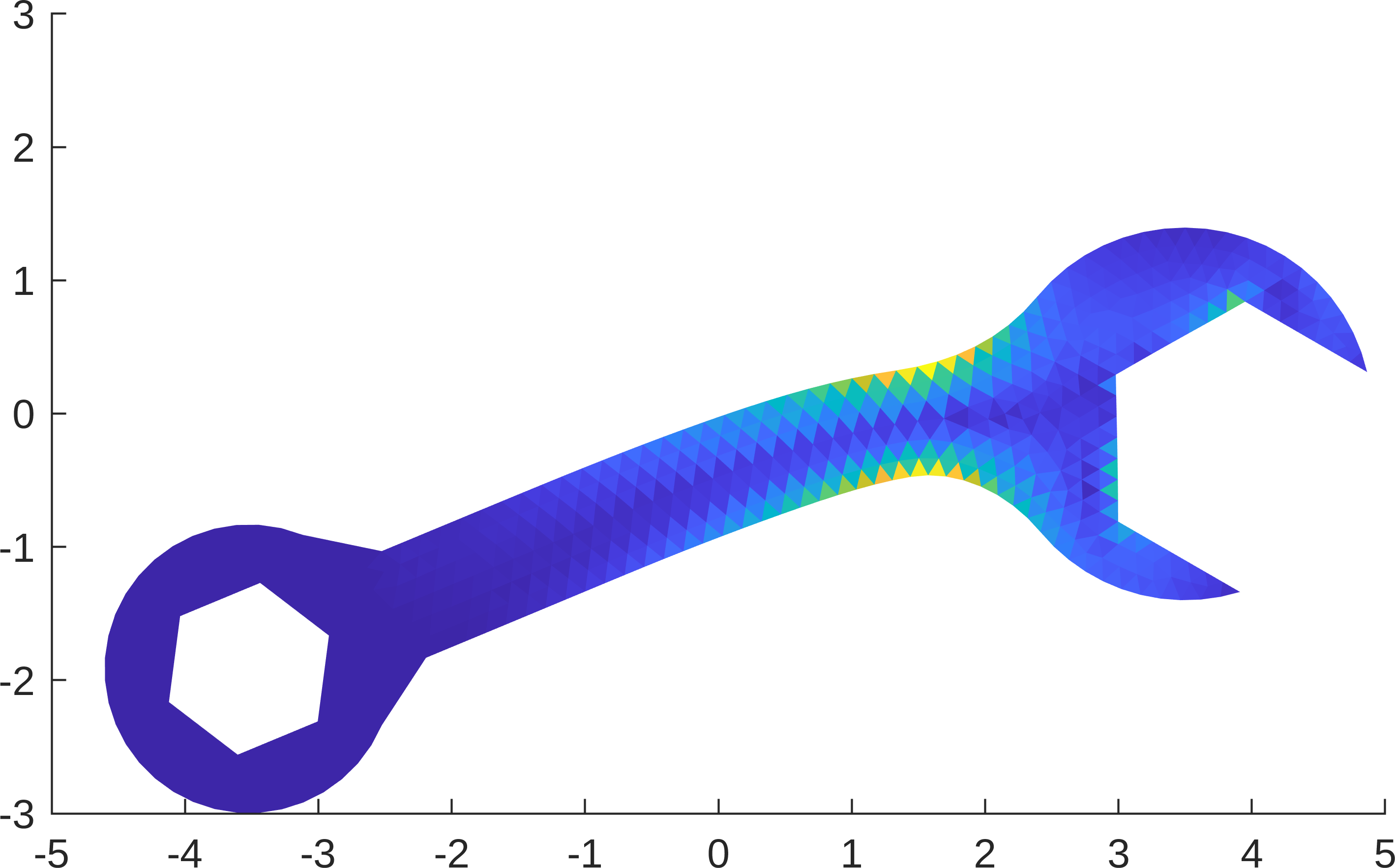}};
    \node at (0,0.25) {\includegraphics[width=0.35\linewidth,height=0.21\linewidth]{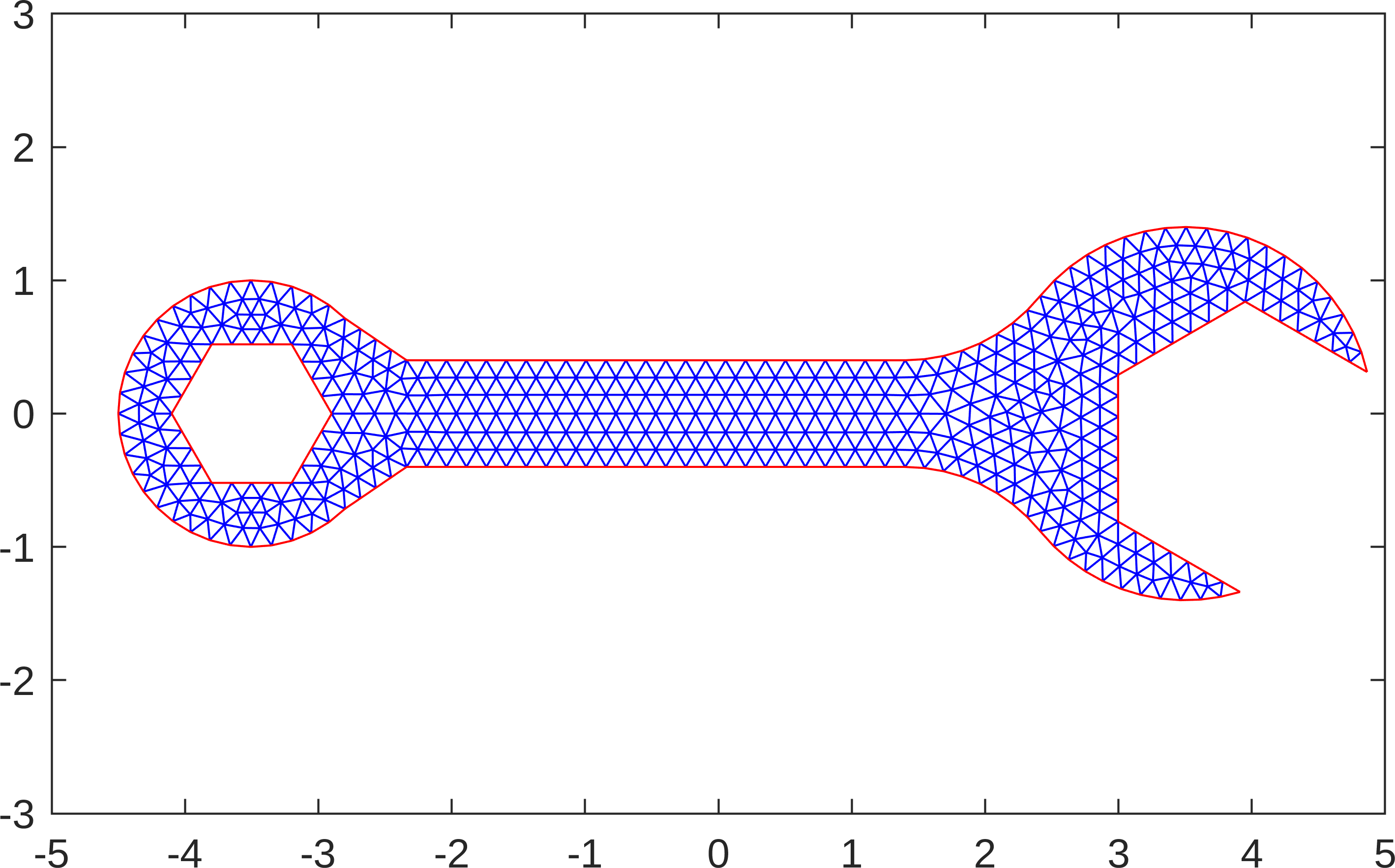}};  
    \node[rotate=-32] at (2.9*0.8,-0.3*0.7) {$u=0$};
    \node[rotate=-32] at (3.2*0.8,0.3*0.7) {$u=0$};
    \draw[->,red,line width=2pt] (-1.18,1.4) -- (-1.18,0.63);
    \draw[->,red,line width=2pt] (0.85,1.4) -- (0.85,0.63);
    \node[red,anchor=west] at (-1.5*0.7,1.3*0.7) {$f {=} [0, -1]^T$};
    \foreach \x in {-2.33824, -2.18871, -2.03918, -1.88965, -1.74012, -1.59059, -1.44106, -1.29153, -1.142, -0.992471, -0.842941, -0.693412, -0.543882, -0.394353, -0.244824, -0.0952941, 0.0542353, 0.203765, 0.353294, 0.502824, 0.652353, 0.801882, 0.951412, 1.10094, 1.25047, 1.4} {
    \node[ellipse,draw=green!50!black,fill=green!50!black,inner sep=0.7pt,anchor=center] at (0.545*\x+0.09,0.55) {};
    }
    \node at (0,-.1) {};
\end{tikzpicture}
\end{center}
\caption{Left: the true log Young’s modulus used for generating observed data. Right: the displacement of the wrench.}\label{fig:wrench}
\end{figure}

\begin{figure}
\begin{center}
    \hspace{-1em}
    \includegraphics[width=0.5\linewidth, trim={0, 9em, 0, 8em}, clip]{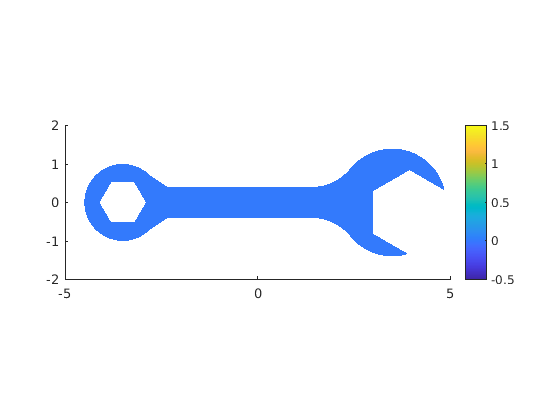}\hspace{-1em}
    \includegraphics[width=0.5\linewidth, trim={0, 9em, 0, 8em}, clip]{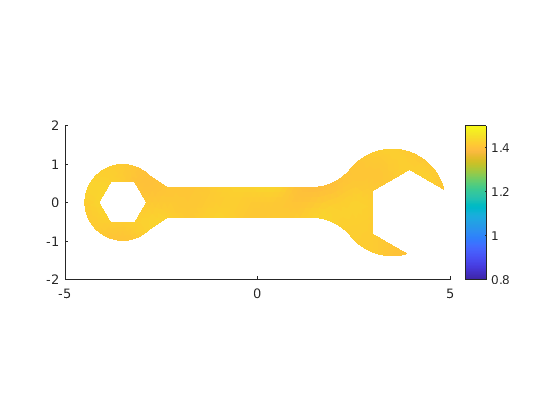}
    \hspace{-2em}
\end{center}

\begin{center}
    \hspace{-1em}
    \includegraphics[width=0.5\linewidth, trim={0, 9em, 0, 8em}, clip]{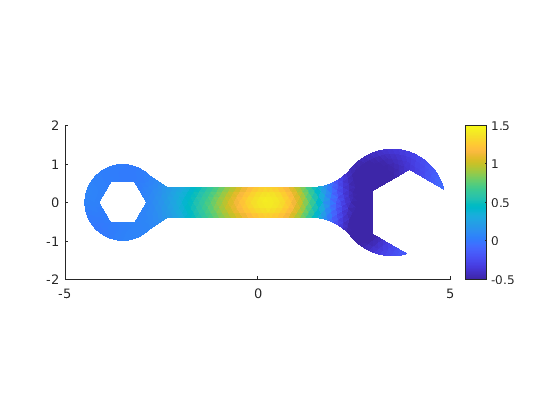}\hspace{-1em}
    \includegraphics[width=0.5\linewidth, trim={0, 9em, 0, 8em}, clip]{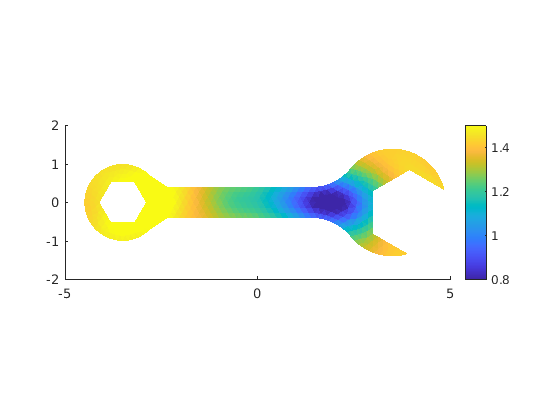}
    \hspace{-2em}
\end{center}
\caption{Summary statistics of the log Young’s modulus. Top left: the prior mean. Top right: the prior standard deviation. Bottom left: the estimated posterior mean. Bottom right: the estimated posterior standard deviation.}\label{fig:wrench_result}
\end{figure}

Because the PDE model used in this example has a rather demanding computational cost, we only demonstrate the performance of Algorithm \ref{alg:pMCMC} using the most efficient MCMC samplers. In this example, NUTS failed to sample the posterior, because the Hamiltonian dynamics push the Markov chain to extreme tails of the posterior, where the numerical discretization of the PDE becomes unstable to solve. 
Recall that the standard full-space pCN and MALA are about three orders of magnitude slower than their subspace counterpart in Example 1 (cf. Table \ref{table:iact_elliptic}). In this example, we will not repeat these trivial accelerations. Instead, we compare H-MALA (which is the most efficient full-space sampling method in Example 1) with the subspace MALA using LIS dimensions $r = \{24, 32, 40\}$ and a pseudo-marginal sample size $m=2$. The IACTs are reported in Table \ref{table:iact_wrench}. Similar to the first example, Algorithm \ref{alg:pMCMC} with LIS is able to significantly accelerate MCMC sampling in this example. 

\begin{table}
    \centering
    \caption{Wrench example. Average IACTs of parameters computed by subspace MALA and H-MALA are reported. All the data reported here are in the form of mean$\pm$standard derivation.}
    \label{table:iact_wrench}
    \begin{tabular}{ll|c|c}
    \hline
    && pseudo-marginal ($m=2$)  &  Hessian-preconditioned \\ \hline
    \multirow{3}{*}{MALA} & $r=24$ & $5.45\!\pm\!0.83$ & \multirow{3}{*}{$67.7\!\pm\!10$} \\
    & $r=32$ & $6.19\!\pm\!0.75$ & \\
    & $r=40$ & $6.58\!\pm\!0.96$ & \\\hline
    \end{tabular}
\end{table}

\section{Conclusion}

In this paper we discussed the design of efficient MCMC algorithms for high-dimensional Bayesian inverse problems with heavy-tailed priors.
Our methodology relies on two building blocks: first, we use \emph{normalizing transformations} in order to reformulate the original inverse problem in a reference space endowed with a Gaussian prior. Then, we project the high dimensional problem onto a suitable \emph{likelihood-informed subspace} (LIS) which is detected using the gradient of the log-likelihood function.
The way we detect the LIS in the transformed (Gaussian) coordinates permits us to control rigorously the error caused by the dimension reduction.
Furthermore, we exploit the LIS by designing efficient MCMC algorithms which better explore the important directions.
While the results from these procedures can be affected by the usage of inaccurate transformations, we can correct the errors through the delayed acceptance approach.
Finally, we demonstrate the effectiveness of these MCMC procedures numerically on an elliptic PDE and a linear elasticity problem.

\section*{Acknowledgements}
TC acknowledges support from the Australian Research Council under the grant DP210103092.
XT is supported by the Singapore Ministry of Education (MOE) grant R-146-000-292-114. 
OZ acknowledges support from the ANR JCJC project MODENA (ANR-21-CE46-0006-01). 

\appendix
\renewcommand{\thesection}{\Alph{section}}
\titleformat{\section}{\bfseries}{\appendixname~\thesection .}{0.5em}{}

\section{Asymptotic behaviour of $T$ for some standard distributions}
\label{sec:Proof_Besov}

In this section we prove that map $T(z)=(\calP^0)^{-1} \circ \Phi^0(z)$ which pushes forward the normal density $\phi^0(z)\propto e^{-z^2/2}$ to $\pi^0(x) \propto e^{-\lambda|x|^p}$ for some $p,\lambda>0$ satisfies
$$
 T(z) \sim \text{sign}(z) \left( \frac{z^2}{2 \lambda} \right)^{1/p}
 \quad\text{and}\quad
 T'(z) \sim \frac{z} { \lambda p} \left( \frac{z^2}{2 \lambda} \right)^{1/p-1} ,
$$
when $z\rightarrow\pm\infty$.
First we give the following lemma.
\begin{lemma}\label{lemma:1}
 Let $p\geq1$ and $\lambda>0$. Then for any $x>0$ we have
\begin{equation}\label{eq:tmp327658}
    \frac{e^{-\lambda x^p}}{\lambda p x^{p-1}}  \left( 1+ \frac{1-p}{\lambda p x^{p}} \right) \leq \int_x^{\infty} e^{-\lambda t^p} \d t \leq \frac{e^{-\lambda x^p}}{\lambda p x^{p-1}} .
\end{equation}
Moreover, for any $1/2\leq p \leq 1$ and $\lambda>0$ and $x>0$, we have 
\begin{equation}\label{eq:tmp327651}
    \frac{e^{-\lambda x^p}}{\lambda p x^{p-1}} \left( 1+ \frac{1-p}{\lambda p x^{p}} + \frac{(1-p)(1-2p)}{(\lambda p x^{p})^2 } \right) 
 \leq \int_x^{\infty} e^{-\lambda t^p} \d t 
 \leq \frac{e^{-\lambda x^p}}{\lambda p x^{p-1}} \left( 1+ \frac{1-p}{\lambda p x^{p}} \right) .
\end{equation}
More generally, for any $n\in\mathbb{N}$ and any $\frac{1}{n+1} \leq p \leq \frac{1}{n}$ we have
\begin{equation}\label{eq:tmp327652}
 \frac{e^{-\lambda x^p}}{\lambda p x^{p-1}} \sum_{i=0}^{n+1} \prod_{k=1}^i \frac{ 1-kp}{\lambda p x^{p}} \leq \int_x^{\infty} e^{-\lambda t^p} \d t 
 \leq \frac{e^{-\lambda x^p}}{\lambda p x^{p-1}} \sum_{i=0}^n \prod_{k=1}^i \frac{ 1-kp}{\lambda p x^{p}} .
\end{equation}
\end{lemma}
\begin{proof}
One integration by part yields
\begin{align}
\notag
    \int_x^{\infty} e^{-\lambda t^p} \d t 
    &= \int_x^{\infty} \frac{ t^{p-1} e^{-\lambda t^p}}{ t^{p-1}} \d t   \nonumber\\
    \notag
    &= \left[ \frac{e^{-\lambda t^p}}{-\lambda p t^{p-1}} \right]_x^\infty - \int_x^{\infty} \frac{ (1-p) e^{-\lambda t^p}}{ -\lambda p t^p} \d t \nonumber\\
    \label{tmp:1}
    &= \frac{e^{-\lambda x^p}}{\lambda p x^{p-1}}  + \frac{(1-p)}{\lambda p}\int_x^{\infty} \frac{  e^{-\lambda t^p}}{  t^p} \d t  . 
\end{align}
Thus, if $p\geq 1$ then $\int_x^{\infty} e^{-\lambda t^p} \d t \leq \frac{e^{-\lambda x^p}}{\lambda p x^{p-1}}$ which gives the right-hand side of \eqref{eq:tmp327658}. On the other hand, $p\geq1$ ensures $2p-1 \geq0$ so that
$$
 \int_x^{\infty} \frac{  e^{-\lambda t^p}}{  t^p} \d t = \int_x^{\infty} \frac{  t^{p-1}e^{-\lambda t^p}}{  t^{2p-1}} \d t 
 \leq \frac{1}{  x^{2p-1}}\int_x^{\infty} t^{p-1}e^{-\lambda t^p} \d t
 = \frac{1}{ x^{2p-1}} \frac{e^{-\lambda x^p}}{\lambda p} ,
$$
and then $ \int_x^{\infty} e^{-\lambda t^p} \d t \geq \frac{e^{-\lambda x^p}}{\lambda p x^{p-1}}  \left( 1+ \frac{1-p}{\lambda p x^{p}} \right) $, which is the left-hand side of \eqref{eq:tmp327658}. This shows that \eqref{eq:tmp327658} holds for any $p\geq 1$.

Now let $1/2\leq p\leq 1$. To show \eqref{eq:tmp327651}, we apply one more integration by part to \eqref{tmp:1} and find
\begin{align*}
   \int_x^{\infty} \frac{  e^{-\lambda t^p}}{  t^p} \d t 
   &= \int_x^{\infty} \frac{  t^{p-1}e^{-\lambda t^p}}{  t^{2p-1}} \d t  \\
   &=\left[ \frac{e^{-\lambda t^p}}{-\lambda p t^{2p-1}} \right]_x^\infty - \int_x^{\infty} \frac{ (1-2p) e^{-\lambda t^p}}{ -\lambda p t^{2p}} \d t \\
   &= \frac{e^{-\lambda x^p}}{\lambda p x^{2p-1}} + (1-2p) \int_x^{\infty} \frac{ e^{-\lambda t^p}}{ \lambda p t^{2p}} \d t ,
\end{align*}
so that
$$
 \int_x^{\infty} e^{-\lambda t^p} \d t  
 = \frac{e^{-\lambda x^p}}{\lambda p x^{p-1}}  + \frac{(1-p)}{(\lambda p)^2}
 \frac{e^{-\lambda x^p}}{ x^{2p-1}} + \frac{(1-p)(1-2p)}{(\lambda p)^2} \int_x^{\infty} \frac{ e^{-\lambda t^p}}{ t^{2p}} \d t .
$$
Because $1/2\leq p\leq 1$ we have $1-p\geq0$ and $(1-p)(1-2p) \leq 0$ so that $ \int_x^{\infty} e^{-\lambda t^p} \d t \leq \frac{e^{-\lambda x^p}}{\lambda p x^{p-1}}  \left( 1+ \frac{1-p}{\lambda p x^{p}} \right) $. This gives the right-hand side of \eqref{eq:tmp327651}.
Furthermore, $1/2\leq p\leq 1$ ensures $3p-1\geq0$ so that
$$
 \int_x^{\infty} \frac{ e^{-\lambda t^p}}{ t^{2p}} \d t 
 = \int_x^{\infty} \frac{ t^{p-1}e^{-\lambda t^p}}{ t^{3p-1}} \d t
 \leq \frac{1}{x^{3p-1}} \int_x^{\infty}  t^{p-1}e^{-\lambda t^p} \d t
 = \frac{1}{x^{3p-1} } \frac{ e^{-\lambda x^p} }{\lambda p}.
$$
so that $ \int_x^{\infty} e^{-\lambda t^p} \d t \geq \frac{e^{-\lambda x^p}}{\lambda p x^{p-1}}  \left( 1+ \frac{1-p}{\lambda p x^{p}} + \frac{(1-p)(1-2p)}{(\lambda p)^2x^{2p}} \right)$.
This shows that \eqref{eq:tmp327651} holds for $1/2\leq p\leq 1$.

By recursion, we have that for any $n\geq1$ the inequality \eqref{eq:tmp327652} holds for any $\frac{1}{n+1}\leq p \leq \frac1n$.

\end{proof}

\begin{lemma}
\label{lem:Besov}
Suppose $\pi^0(x)= \mathcal{Z}_{p,\lambda}^{-1} \exp(-\lambda |x|^p)$ with $p,\lambda>0$ and $\mathcal{Z}_{p,\lambda}=\int\exp(-\lambda |x|^p)\d x$. Then 
\[
T(z) \sim \left( \frac{z^2}{2 \lambda} \right)^{1/p},\qquad T'(z) \sim  \frac{z} { \lambda p} \left( \frac{z^2}{2 \lambda} \right)^{1/p-1}  ,
\]
when $z\to \infty$. 
\end{lemma}
\begin{proof}
For any $x>0$, we have $ \calP^0(x) = 1-\mathcal{Z}_{p,\lambda}^{-1}\int_{x}^{\infty} e^{-\lambda x^p}\d x$ so that Lemma \ref{lemma:1} ensures
\begin{equation}\label{eq:tmp235098731}
 1- e^{-\lambda x^p}F_n(x)
 \leq \calP^0(x) \leq 
 1- e^{-\lambda x^p}F_{n+1}(x) .
\end{equation}
Here, $n$ is the integer such that $\frac{1}{n+1}\leq p \leq \frac1n$ and $F_n(x)$ is given by
$$
F_n(x) = \frac{1}{\mathcal{Z}_{p,\lambda}\lambda p x^{p-1}} \sum_{i=0}^{n} \prod_{k=1}^i \frac{ 1-kp}{\lambda p x^p} .
$$
Because $F_n(x)$ is a rational function in $x$, it is dominated by the exponential function. Thus, we have that for any $\varepsilon>0$, there exists $A\geq0$ such that $F_n(x) \leq e^{\lambda \varepsilon x^p}$ and $F_{n+1}(x) \geq e^{-\lambda \varepsilon x^p}$ hold for any $x\geq A$.
We deduce that
$$
 1-e^{-\lambda(1-\varepsilon) x^p}
 \leq \calP^0(x) \leq 
 1-e^{-\lambda(1+\varepsilon) x^p} ,
$$
holds for any $x\geq A$. Replacing $x$ with $T(z)=(\calP^0)^{-1}\circ \Phi^0(z)$ we obtain
$$
 -\frac{\log (1-\Phi^0(z))}{\lambda (1+\varepsilon)}
 \leq T(z)^p \leq
 -\frac{\log (1-\Phi^0(z))}{\lambda (1-\varepsilon)} ,
$$
for any $z\geq T^{-1}(A)$.
Because $1-\Phi^0(z) = \frac{1}{\sqrt{2\pi}}\int_{z}^\infty e^{-t^2/2}\d t $, Lemma \ref{lemma:1} with $p=2$ and $\lambda=1/2$ ensures $ \frac{e^{-z^2/2}}{z\sqrt{2\pi}} (1-1/z^2)\leq 1-\Phi^0(z) \leq \frac{e^{-z^2/2}}{z\sqrt{2\pi}} $. Thus, 
\begin{equation*}
 \frac{ z^2 }{2\lambda (1+\varepsilon)}
 +\frac{\log (z\sqrt{2\pi})}{\lambda (1+\varepsilon)}
 \leq T(z)^p \leq
 \frac{ z^2 }{2\lambda (1-\varepsilon)}
 +\frac{\log(z\sqrt{2\pi}/(1-1/z^2))}{\lambda (1-\varepsilon)} ,
\end{equation*}
for any $z\geq T^{-1}(A)$. We deduce that 
$ T(z) \sim ( \frac{z^2}{2 \lambda} )^{1/p} $ when $z\rightarrow\infty$.
Because $T(-z)= -T(z)$, we obtain $ T(z) \sim -( \frac{z^2}{2 \lambda} )^{1/p} $ when $z\rightarrow-\infty$.

We now analyze $T'(z)$. Assume $x>0$. 
By letting $x=T(z)$ in \eqref{eq:tmp235098731} we obtain
$$
 \frac{T(z)^{p-1} \lambda p( 1- \Phi^0(x)) }{\sum_{i=0}^{n} \prod_{k=1}^i \frac{ 1-kp}{\lambda p T(z)^p} } 
 \leq  \frac{e^{-\lambda T(z)^p}}{\mathcal{Z}_{p,\lambda}}  \leq 
 \frac{T(z)^{p-1} \lambda p( 1- \Phi^0(x)) }{ \sum_{i=0}^{n+1} \prod_{k=1}^i \frac{ 1-kp}{\lambda p T(z)^p}} .
$$
Since $T'(z)=\frac{\phi^0(z)}{\pi^0(T(z))}$, we can write
$$
 \frac{\phi^0(z)\sum_{i=0}^{n+1} \prod_{k=1}^i \frac{ 1-kp}{\lambda p T(z)^p} } {T(z)^{p-1} \lambda p( 1- \Phi^0(x)) }
 \leq  T'(z) \leq 
 \frac{\phi^0(z) \sum_{i=0}^{n} \prod_{k=1}^i \frac{ 1-kp}{\lambda p T(z)^p}}{T(z)^{p-1} \lambda p( 1- \Phi^0(x)) } .
$$
so that, using the fact that $\frac{\phi^0(z)}{z}(1-1/z^2)\leq 1-\Phi^0(z)\leq \frac{\phi^0(z)}{z}$, we get
$$
 \frac{z}{T(z)^{p-1} \lambda p } \underbrace{ \sum_{i=0}^{n+1} \prod_{k=1}^i \frac{ 1-kp}{\lambda p T(z)^p} }_{\rightarrow 1}
 \leq  T'(z) \leq 
 \frac{z}{T(z)^{p-1} \lambda p }\underbrace{\frac{\sum_{i=0}^{n} \prod_{k=1}^i \frac{ 1-kp}{\lambda p T(z)^p}} {1-1/z^2}}_{\rightarrow 1}.
$$
We deduce that
$$
 T'(z) \sim \frac{z}{T(z)^{p-1} \lambda p} \sim \frac{z} { \lambda p} \left( \frac{z^2}{2 \lambda} \right)^{1/p-1} 
$$
when $z\rightarrow\infty$. Because $T'(-z)=T'(z)$, we deduce that the above equivalent also holds when $z\rightarrow-\infty$.

\end{proof}

\section*{Acknowledgments}

\bibliographystyle{plain}
\bibliography{ref}

\begin{thebibliography}{10}

\bibitem{andrieu2009pseudo}
Christophe Andrieu and Gareth~O Roberts.
\newblock The pseudo-marginal approach for efficient monte carlo computations.
\newblock {\em The Annals of Statistics}, 37(2):697--725, 2009.

\bibitem{andrieu2015convergence}
Christophe Andrieu and Matti Vihola.
\newblock Convergence properties of pseudo-marginal markov chain monte carlo
  algorithms.
\newblock {\em The Annals of Applied Probability}, 25(2):1030--1077, 2015.

\bibitem{atchade2006adaptive}
Yves~F Atchad{\'e}.
\newblock An adaptive version for the metropolis adjusted langevin algorithm
  with a truncated drift.
\newblock {\em Methodology and Computing in applied Probability},
  8(2):235--254, 2006.

\bibitem{baptista2021learning}
Ricardo Baptista, Youssef Marzouk, Rebecca~E Morrison, and Olivier Zahm.
\newblock Learning non-gaussian graphical models via hessian scores and
  triangular transport.
\newblock {\em arXiv preprint arXiv:2101.03093}, 2021.

\bibitem{baptista2020adaptive}
Ricardo Baptista, Youssef Marzouk, and Olivier Zahm.
\newblock On the representation and learning of monotone triangular transport
  maps.
\newblock {\em arXiv preprint arXiv:2009.10303}, 2020.

\bibitem{baptista2022gradient}
Ricardo Baptista, Youssef Marzouk, and Olivier Zahm.
\newblock Gradient-based data and parameter dimension reduction for bayesian
  models: an information theoretic perspective.
\newblock {\em arXiv preprint arXiv:2207.08670}, 2022.

\bibitem{bardsley2021optimization}
Johnathan~M Bardsley and Tiangang Cui.
\newblock Optimization-based markov chain monte carlo methods for nonlinear
  hierarchical statistical inverse problems.
\newblock {\em SIAM/ASA Journal on Uncertainty Quantification}, 9(1):29--64,
  2021.

\bibitem{beskos2008mcmc}
Alexandros Beskos, Gareth Roberts, Andrew Stuart, and Jochen Voss.
\newblock Mcmc methods for diffusion bridges.
\newblock {\em Stochastics and Dynamics}, 8(03):319--350, 2008.

\bibitem{bhadra2019lasso}
Anindya Bhadra, Jyotishka Datta, Nicholas~G Polson, and Brandon Willard.
\newblock Lasso meets horseshoe: A survey.
\newblock {\em Statistical Science}, 34(3):405--427, 2019.

\bibitem{bigoni2021nonlinear}
Daniele Bigoni, Youssef Marzouk, Clémentine Prieur, and Olivier Zahm.
\newblock Nonlinear dimension reduction for surrogate modeling using gradient
  information.
\newblock {\em Information and Inference: A Journal of the IMA}, 05 2022.

\bibitem{bogachev2005triangular}
Vladimir~Igorevich Bogachev, Aleksandr~Viktorovich Kolesnikov, and
  Kirill~Vladimirovich Medvedev.
\newblock Triangular transformations of measures.
\newblock {\em Sbornik: Mathematics}, 196(3):309, 2005.

\bibitem{bui2014solving}
Tan Bui-Thanh and Mark Girolami.
\newblock Solving large-scale pde-constrained bayesian inverse problems with
  riemann manifold hamiltonian monte carlo.
\newblock {\em Inverse Problems}, 30(11):114014, 2014.

\bibitem{carvalho2009handling}
Carlos~M Carvalho, Nicholas~G Polson, and James~G Scott.
\newblock Handling sparsity via the horseshoe.
\newblock In {\em Artificial Intelligence and Statistics}, pages 73--80. PMLR,
  2009.

\bibitem{chen2018robust}
Victor Chen, Matthew~M Dunlop, Omiros Papaspiliopoulos, and Andrew~M Stuart.
\newblock Robust mcmc sampling with non-gaussian and hierarchical priors in
  high dimensions.
\newblock {\em arXiv preprint arXiv:1803.03344}, 3, 2018.

\bibitem{christen2005markov}
J~Andr{\'e}s Christen and Colin Fox.
\newblock Markov chain monte carlo using an approximation.
\newblock {\em Journal of Computational and Graphical statistics},
  14(4):795--810, 2005.

\bibitem{cotter2013mcmc}
Simon~L Cotter, Gareth~O Roberts, Andrew~M Stuart, and David White.
\newblock Mcmc methods for functions: modifying old algorithms to make them
  faster.
\newblock {\em Statistical Science}, pages 424--446, 2013.

\bibitem{cui2021conditional}
Tiangang Cui, Sergey Dolgov, and Olivier Zahm.
\newblock Conditional deep inverse rosenblatt transports.
\newblock {\em arXiv preprint arXiv:2106.04170}, 2021.

\bibitem{Cuietal16b}
Tiangang Cui, Kody J.~H. Law, and Youssef~M. Marzouk.
\newblock Dimension-independent likelihood-informed {MCMC}.
\newblock {\em J. Comput. Phys.}, 304:109--137, 2016.

\bibitem{cui2014likelihood}
Tiangang Cui, James Martin, Youssef~M Marzouk, Antti Solonen, and Alessio
  Spantini.
\newblock Likelihood-informed dimension reduction for nonlinear inverse
  problems.
\newblock {\em Inverse Problems}, 30(11):114015, 2014.

\bibitem{EtAl1Cui6}
Tiangang Cui, Youssef Marzouk, and Karen Willcox.
\newblock Scalable posterior approximations for large-scale {B}ayesian inverse
  problems via likelihood-informed parameter and state reduction.
\newblock {\em J. Comput. Phys.}, 315:363--387, 2016.

\bibitem{cui2021unified}
Tiangang Cui and Xin~T Tong.
\newblock A unified performance analysis of likelihood-informed subspace
  methods.
\newblock {\em Bernoulli}, 2021.

\bibitem{cui2020data}
Tiangang Cui and Olivier Zahm.
\newblock Data-free likelihood-informed dimension reduction of bayesian inverse
  problems.
\newblock {\em Inverse Problems}, 37(4):045009, 2021.

\bibitem{dashti2012besov}
Masoumeh Dashti, Stephen Harris, and Andrew Stuart.
\newblock {B}esov priors for {B}ayesian inverse problems.
\newblock {\em Inverse Problems \& Imaging}, 6(2):183, 2012.

\bibitem{fleischer2007transformations}
Mark Fleischer.
\newblock Transformations for accelerating mcmc simulations with broken
  ergodicity.
\newblock In {\em 2007 Winter Simulation Conference}, pages 658--666. IEEE,
  2007.

\bibitem{girolami2011riemann}
Mark Girolami and Ben Calderhead.
\newblock Riemann manifold langevin and hamiltonian monte carlo methods.
\newblock {\em Journal of the Royal Statistical Society: Series B (Statistical
  Methodology)}, 73(2):123--214, 2011.

\bibitem{goodfellow2014generative}
Ian Goodfellow, Jean Pouget-Abadie, Mehdi Mirza, Bing Xu, David Warde-Farley,
  Sherjil Ozair, Aaron Courville, and Yoshua Bengio.
\newblock Generative adversarial nets.
\newblock {\em Advances in neural information processing systems}, 27, 2014.

\bibitem{haario2001adaptive}
Heikki Haario, Eero Saksman, Johanna Tamminen, et~al.
\newblock An adaptive metropolis algorithm.
\newblock {\em Bernoulli}, 7(2):223--242, 2001.

\bibitem{hoffman2014no}
Matthew~D Hoffman, Andrew Gelman, et~al.
\newblock The no-u-turn sampler: adaptively setting path lengths in hamiltonian
  monte carlo.
\newblock {\em J. Mach. Learn. Res.}, 15(1):1593--1623, 2014.

\bibitem{hosseini2017well}
Bamdad Hosseini.
\newblock Well-posed bayesian inverse problems with infinitely divisible and
  heavy-tailed prior measures.
\newblock {\em SIAM/ASA Journal on Uncertainty Quantification},
  5(1):1024--1060, 2017.

\bibitem{kallenberg1997foundations}
Olav Kallenberg and Olav Kallenberg.
\newblock {\em Foundations of modern probability}, volume~2.
\newblock Springer, 1997.

\bibitem{kokiopoulou2011trace}
Effrosini Kokiopoulou, Jie Chen, and Yousef Saad.
\newblock Trace optimization and eigenproblems in dimension reduction methods.
\newblock {\em Numerical Linear Algebra with Applications}, 18(3):565--602,
  2011.

\bibitem{lam2020multifidelity}
Remi Lam, Olivier Zahm, Youssef Marzouk, and Karen Willcox.
\newblock Multifidelity dimension reduction via active subspaces.
\newblock {\em SIAM Journal on Scientific Computing}, 42(2):A929--A956, 2020.

\bibitem{saksman2009discretization}
Matti Lassas, Eero Saksman, and Samuli Siltanen.
\newblock Discretization-invariant {B}ayesian inversion and {B}esov space
  priors.
\newblock {\em Inverse problems and imaging}, 3(1):87--122, 2009.

\bibitem{lebrun2009innovating}
Regis Lebrun and Anne Dutfoy.
\newblock An innovating analysis of the nataf transformation from the copula
  viewpoint.
\newblock {\em Probabilistic Engineering Mechanics}, 24(3):312--320, 2009.

\bibitem{lemaire2013structural}
Maurice Lemaire.
\newblock {\em Structural reliability}.
\newblock John Wiley \& Sons, 2013.

\bibitem{liu1998sequential}
Jun~S Liu and Rong Chen.
\newblock Sequential monte carlo methods for dynamic systems.
\newblock {\em Journal of the American statistical association},
  93(443):1032--1044, 1998.

\bibitem{ma2019sampling}
Yi-An Ma, Yuansi Chen, Chi Jin, Nicolas Flammarion, and Michael~I Jordan.
\newblock Sampling can be faster than optimization.
\newblock {\em Proceedings of the National Academy of Sciences},
  116(42):20881--20885, 2019.

\bibitem{majda2015intermittency}
Andrew~J Majda and Xin~T Tong.
\newblock Intermittency in turbulent diffusion models with a mean gradient.
\newblock {\em Nonlinearity}, 28(11):4171, 2015.

\bibitem{majda2019simple}
Andrew~J Majda and Xin~T Tong.
\newblock Simple nonlinear models with rigorous extreme events and heavy tails.
\newblock {\em Nonlinearity}, 32(5):1641, 2019.

\bibitem{markkanen2019cauchy}
Markku Markkanen, Lassi Roininen, Janne~MJ Huttunen, and Sari Lasanen.
\newblock Cauchy difference priors for edge-preserving bayesian inversion.
\newblock {\em Journal of Inverse and Ill-posed Problems}, 27(2):225--240,
  2019.

\bibitem{martin2012stochastic}
James Martin, Lucas~C Wilcox, Carsten Burstedde, and Omar Ghattas.
\newblock A stochastic newton mcmc method for large-scale statistical inverse
  problems with application to seismic inversion.
\newblock {\em SIAM Journal on Scientific Computing}, 34(3):A1460--A1487, 2012.

\bibitem{nataf1962determination}
Andre Nataf.
\newblock Determination des distribution don t les marges sont donnees.
\newblock {\em Comptes Rendus de l Academie des Sciences}, 225:42--43, 1962.

\bibitem{neal2011mcmc}
Radford~M Neal et~al.
\newblock Mcmc using hamiltonian dynamics.
\newblock {\em Handbook of markov chain monte carlo}, 2(11):2, 2011.

\bibitem{Owen}
A.B. Owen.
\newblock {\em Monte Carlo Theory, Methods and Examples}.
\newblock 2013.

\bibitem{parno2018transport}
Matthew~D Parno and Youssef~M Marzouk.
\newblock Transport map accelerated markov chain monte carlo.
\newblock {\em SIAM/ASA Journal on Uncertainty Quantification}, 6(2):645--682,
  2018.

\bibitem{petra2014computational}
Noemi Petra, James Martin, Georg Stadler, and Omar Ghattas.
\newblock A computational framework for infinite-dimensional bayesian inverse
  problems, part ii: Stochastic newton mcmc with application to ice sheet flow
  inverse problems.
\newblock {\em SIAM Journal on Scientific Computing}, 36(4):A1525--A1555, 2014.

\bibitem{peyre2019computational}
Gabriel Peyr{\'e}, Marco Cuturi, et~al.
\newblock Computational optimal transport: With applications to data science.
\newblock {\em Foundations and Trends{\textregistered} in Machine Learning},
  11(5-6):355--607, 2019.

\bibitem{rezende2015variational}
Danilo Rezende and Shakir Mohamed.
\newblock Variational inference with normalizing flows.
\newblock In {\em International conference on machine learning}, pages
  1530--1538. PMLR, 2015.

\bibitem{robert2021rao}
Christian~P Robert and Gareth~O Roberts.
\newblock Rao-blackwellization in the mcmc era.
\newblock {\em arXiv preprint arXiv:2101.01011}, 2021.

\bibitem{roberts2001optimal}
Gareth~O Roberts and Jeffrey~S Rosenthal.
\newblock Optimal scaling for various metropolis-hastings algorithms.
\newblock {\em Statistical science}, 16(4):351--367, 2001.

\bibitem{roberts1996exponential}
Gareth~O Roberts and Richard~L Tweedie.
\newblock Exponential convergence of langevin distributions and their discrete
  approximations.
\newblock {\em Bernoulli}, pages 341--363, 1996.

\bibitem{saltelli2008global}
Andrea Saltelli, Marco Ratto, Terry Andres, Francesca Campolongo, Jessica
  Cariboni, Debora Gatelli, Michaela Saisana, and Stefano Tarantola.
\newblock {\em Global sensitivity analysis: the primer}.
\newblock John Wiley \& Sons, 2008.

\bibitem{smetana2020randomized}
Kathrin Smetana and Olivier Zahm.
\newblock Randomized residual-based error estimators for the proper generalized
  decomposition approximation of parametrized problems.
\newblock {\em International Journal for Numerical Methods in Engineering},
  121(23):5153--5177, 2020.

\bibitem{spantini2018inference}
Alessio Spantini, Daniele Bigoni, and Youssef Marzouk.
\newblock Inference via low-dimensional couplings.
\newblock {\em The Journal of Machine Learning Research}, 19(1):2639--2709,
  2018.

\bibitem{Stuart10}
A.M. Stuart.
\newblock Inverse problems: a {B}ayesian perspective.
\newblock {\em Acta Numer.}, 19:451--559, 2010.

\bibitem{sullivan2017well}
TJ~Sullivan.
\newblock Well-posedness of bayesian inverse problems in quasi-banach spaces
  with stable priors.
\newblock {\em PAMM}, 17(1):871--874, 2017.

\bibitem{chada2021cauchy}
Jarkko Suuronen, Neil~K Chada, and Lassi Roininen.
\newblock Cauchy markov random field priors for bayesian inversion.
\newblock {\em Statistics and Computing}, 32(2):1--26, 2022.

\bibitem{unser2014introduction}
Michael Unser and Pouya~D Tafti.
\newblock {\em An introduction to sparse stochastic processes}.
\newblock Cambridge University Press, 2014.

\bibitem{villani2009optimal}
C{\'e}dric Villani.
\newblock {\em Optimal transport: old and new}, volume 338.
\newblock Springer, 2009.

\bibitem{wang2017bayesian}
Zheng Wang, Johnathan~M Bardsley, Antti Solonen, Tiangang Cui, and Youssef~M
  Marzouk.
\newblock Bayesian inverse problems with l\_1 priors: a randomize-then-optimize
  approach.
\newblock {\em SIAM Journal on Scientific Computing}, 39(5):S140--S166, 2017.

\bibitem{yao2016tv}
Zhewei Yao, Zixi Hu, and Jinglai Li.
\newblock A tv-gaussian prior for infinite-dimensional bayesian inverse
  problems and its numerical implementations.
\newblock {\em Inverse Problems}, 32(7):075006, 2016.

\bibitem{zahm2020gradient}
Olivier Zahm, Paul~G Constantine, Clementine Prieur, and Youssef~M Marzouk.
\newblock Gradient-based dimension reduction of multivariate vector-valued
  functions.
\newblock {\em SIAM Journal on Scientific Computing}, 42(1):A534--A558, 2020.

\bibitem{zahm2018certified}
Olivier Zahm, Tiangang Cui, Kody Law, Alessio Spantini, and Youssef Marzouk.
\newblock Certified dimension reduction in nonlinear bayesian inverse problems.
\newblock {\em Mathematics of Computation}, 91(336):1789--1835, 2022.

\bibitem{zienkiewicz2000finite}
Olgierd~Cecil Zienkiewicz, Robert~Leroy Taylor, and Robert~Leroy Taylor.
\newblock {\em The finite element method: solid mechanics}, volume~2.
\newblock Butterworth-heinemann, 2000.

\end{thebibliography}

\end{document}